\documentclass[12pt]{amsart}
\usepackage{polski}
\usepackage[T1]{fontenc}
\usepackage{amsmath, amsthm, amssymb, amsfonts, enumerate}
\usepackage{enumitem}
\usepackage[colorlinks=true,linkcolor=blue,urlcolor=blue]{hyperref}
\usepackage{dsfont}  
\usepackage{color}
\usepackage{geometry}
\usepackage{todonotes}
\usepackage[thinc]{esdiff}

\usepackage{enumerate}
\newtheorem{theorem}{Theorem}[section]
\newtheorem{remark}[theorem]{Remark}

\newtheorem{lemma}[theorem]{Lemma}

\newtheorem{corollary}[theorem]{Corollary}
\newtheorem{definition}[theorem]{Definition}

\theoremstyle{plain}

\newcommand{\field}[1]{\mathbb{#1}}

\newcommand{\R}{\field{R}}

\def \and{\quad \text{and} \quad}

\usepackage[utf8]{inputenc}
\usepackage[english]{babel}
\usepackage{indentfirst}
\usepackage{graphicx}
\usepackage{pdfpages}
\usepackage[title]{appendix}
\usepackage{fancyhdr}

\usepackage{enumitem}
\usepackage{amsmath}
\usepackage{amssymb} 
\usepackage{bm}
\usepackage{mathrsfs}
\usepackage{amsthm}
\usepackage{comment}
\usepackage{hyperref}
\usepackage{dsfont} 

\usepackage{mathtools}
\DeclareMathOperator*{\Max}{Max}

\DeclareMathOperator*{\esssup}{ess\,sup}
\DeclareMathOperator*{\essinf}{ess\,inf}

\usepackage{graphicx}
\usepackage{caption}
\usepackage{subcaption}

\title[]{Optimal consumption and investment under relative performance criteria with Epstein-Zin utility}
\author{Jodi Dianetti \and Frank Riedel \and Lorenzo Stanca}
\date{\today}

\begin{document}

\begin{abstract} 
    We consider the strategic interaction of traders in a continuous-time financial market with Epstein-Zin-type recursive intertemporal preferences and performance concerns. We derive explicitly an equilibrium for the finite player and the mean-field version of the game, based on a study of geometric backward stochastic differential equations of Bernoulli type that describe the best replies of traders.   
    Our results show that Epstein-Zin preferences can lead to substantially different equilibrium behavior.
\end{abstract}
\maketitle



\section{Introduction}

Casual observations as well as empirical   evidence   suggest that relative performance concerns   play a significant role in the  decision-making of traders. Many  fund managers are evaluated based on their performance relative to a benchmark index or their peers, creating pressure to match or exceed the performance of others in the industry. 

In this paper, we consider the strategic interaction of such traders with performance concerns when intertemporal preferences are recursive, of the Epstein-Zin type (\cite{duffie1992stochastic}, \cite{epstein1989substitution}), thus extending the previous work of 
\cite{espinosa2015optimal}, \cite{lacker2019mean} and \cite{lacker.soret2020} to this important class of intertemporal preferences.
 The time-additive discounted utility model   is restrictive in many senses.  In particular, it does not allow to disentangle the conceptually and empirically  different 
concepts of risk aversion and intertemporal elasticity of substitution. Epstein-Zin preferences are among the few tractable versions of stochastic differential utility that allow to make this distinction. This disentangling allows for the generation of quantitatively realistic aggregate risk premia (as demonstrated, for example, in the recent work by \cite{brunnermeier2014safe}).

This game's Nash equilibrium can be derived in closed-form despite the intricate interplay between recursive preferences, continuous-time financial markets, and relative performance concerns (see Theorem \ref{thm existence NE}).  We are also able to consider the mean-field version of the game involving potentially a continuum of players (see Theorem \ref{thmfg}). Our study is the first example of a mean-field game involving stochastic differential utility  functions.

Allowing for recursive preferences has important consequences for equilibrium behavior. We show that assuming time-additive utilities might lead to quite misleading conclusions (see Section \ref{econdiscuss}). For example, assuming a usual parameter of relative risk aversion above $1$, one implicitly assumes a rate of intertemporal elasticity of substitution smaller than 1, while empirically one frequently observes a rate of intertemporal elasticity of substitution above 1, see the discussion in \cite[p.~574]{ju2012ambiguity} and \cite{bansal2004risks}. We show that the intertemporal consumption pattern can be completely reversed when one allows for this distinction.

We also show that the parameter of  risk aversion alone determines portfolio choice. The rate of intertemporal elasticity of substitution plays a more significant role in determining consumption patterns, in line with the literature on Epstein-Zin preferences for single agents (e.g., see \cite{kraft.Seiferling.Seifried.2017optimal}).
 
Our problem embeds into stochastic differential games, that are popular models describing competition in a random environment, with countless applications in finance and economics. \cite{HuangMalhameCaines06, LasryLions07} introduced the mean-field game as the limit model when the number of players goes to infinity, thus providing tools to construct approximate Nash equilibria in games involving a large number of players.
Several approaches exist to solve such games, including systems of partial differential equations and of forward-backward stochastic differential equations, see the textbooks \cite{Carmona2016, CarmonaDelarue18} for a detailed description.

Despite the extensive literature and the general abstract existence and characterization results, many approaches encounter computational challenges due to the high dimensionality of the involved equations. Consequently, numerical analysis of equilibria remains highly problematic, even for scenarios with a limited number of players. This underscores the significance of the few explicitly solvable models in the literature and highlights the importance of discovering new ones.

From the methodological point of view, our approach is based on the analysis of systems of backward stochastic differential equations with Bernoulli driver, that we call Bernoulli BSDEs, see Section \ref{SecBernoulliBSDE}. Best replies in our game can be expressed in terms of the solutions to such Bernoulli BSDEs and we are thus able to derive a Nash equilibrium for the finite player and mean-field game that are unique in the class of simple (deterministic) strategies. Indeed, optimization problems with Epstein-Zin recursive utility are known to be related to Bernoulli ordinary differential equations or to partial differential equations with some terms of Bernoulli type as in \cite{kraft.Seiferling.Seifried.2017optimal}. 

In the game context, the optimization problem of the one player is parameterized by the actions of its opponents, and the resulting optimization problem is expressed in terms of a Bernoulli BSDE which does not reduce to an ODE. Despite Bernoulli BSDEs having no Lipschitz driver, our explicit analysis allows us to show that these equations can successfully be used to demonstrate the existence of the equilibria as well as to recover the usual convergence and approximation results relating Nash equilibria and mean field game equilibria, thus justifying the mean-field game as the limit of the finite player game.

%We also show (Theorem \ref{thm convergence}) that the mean-field game is the limit of the finite player games as the number of players grows. 
%The result is non-trivial as the involved backward stochastic differential equations describing intertemporal preferences have drivers that are not Lipschitz-continuous. We thus extend the existing limit theorems for mean-field games as well. 

We finally underline the nature of our mean-field game equilibrium.
When the noises affecting players' decisions are correlated, a common noise appears in the limiting mean-field game and technical challenges arise. Indeed, the equilibrium actions become (in general) only conditionally independent of the future realization of the common noise (see \cite{CarmonaDelarueLacker16}), and only few cases are known in which these are actually adapted to the common source of randomness (see e.g. \cite{ahuja.ren.yang.2019, CarmonaDelarueLacker16, dianetti2022strong, lacker.soret2020} among others). 
Our explicit analysis allows to find  an equilibrium of the latter type.

\subsection*{Related literature}  
We consider the dynamic  problem of consumption and portfolio choice  formalized and studied in the landmark papers of Merton \cite{merton1969,merton1975}. Other papers
 such as \cite{chan2002catching} that incorporated multiple agents into the Merton model,  did so in a general equilibrium
context; in contrast, in our work agents are price-takers in our model, and we do not attempt
to incorporate price equilibrium. Interaction between agents in our model  comes from  a mean field interaction through both the
states and controls (see \cite[Vol.\ I, Chapter 4.6]{CarmonaDelarue18}). This approach has been developed by a recent literature that considered the case of standard utility  \cite{lacker.soret2020,lacker2019mean} (see also \cite{liang2024mean} for the case of habit formation without common noise). Our novelty relies on building on the literature on dynamic portfolio choice problems  with stochastic differential utility started by  \cite{duffie1992stochastic}  (see \cite{seiferling.seifried.2016,kraft.Seiferling.Seifried.2017optimal,belak.al2017,kraft2022bequest} for more recent papers in the literature we build on).

As fas ar the more mathematical literature is concerned, our study belongs to the class of stochastic differential and mean-field games (\cite{HuangMalhameCaines06, LasryLions07}, \cite{Carmona2016, CarmonaDelarue18}).
Solvable games of major relevance are essentially of two types. In  linear-quadratic games,   the equilibria correspond to   solutions of systems of Riccati equations. While an extensive literature addresses these games (see \cite{Carmona2016} and the references therein), recent applications in finance involved the systemic risk analysis of banking networks (see \cite{carmona.fouque.sun.2013systemic.risk}). 
The other class (which is more similar to our model) is the case in which dynamics are geometric and utility functions exhibit constant relative risk aversion type. 
These models were studied in continuous time frameworks in the quite recent papers \cite{fu2023mean, fu.chao2023mean, lacker.soret2020, lacker2019mean}, for both finite player games and mean field games, in order to address portfolio optimization problems for competitive agents.
{In particular, we  refer to  \cite{fu2023mean, fu.chao2023mean} for results on the uniqueness of the equilibrium  and to \cite{bauerle2023nash, bauerle2024nash, goll2024expected} for a discussion of several variants of these models.}
Within this framework, our work shows that the relevant case of games with geometric dynamics and Epstein-Zin utility is still explicitly solvable. Whether the same could be true for linear-quadratic models combined with suitable types of recursive utility remains an open problem that we address with future research. More generally, our work suggests that it is possible to develop a mean field theory for stochastic differential games with recursive utility.

\subsection*{Structure}   The next section describes the model and presents the main results on the finite player and mean-field games, including a discussion of the economic relevance of our results. Section \ref{SecBernoulliBSDE} contains the independent results on geometric Bernoulli Backward Stochastic Differential Equations. Section \ref{section Proof of the Main Theorems} is devoted to the proofs of the main results. Section \ref{endiscuss} provides concluding remarks.

\section{Main Results}

% We begin by describing the preferences of an investor with recursive Epstein-Zin utility and relative performance concerns.

To start with, we introduce  the aggregator and the bequest function that will characterize the intertemporal preferences of our agents.
For a discount rate  $\eta >0 $, relative risk aversion $\gamma>0$,  elasticity of intertemporal substitution $\delta>0$, and    a weight of  bequest utility  $\epsilon >0$  we assume throughout
% \begin{equation}
%     \label{eq parameter q}
%  \delta \not=1, \quad \lambda:=\frac{1-\gamma}{1-\frac{1}{\delta}} \neq 0, \quad q:=1-\frac{1}{\lambda} \ne 0,1, \quad 
%  \text{and either $\gamma \delta, \, \delta \geq 1 $ or $\gamma \delta, \, \delta \leq 1 $.}
% \end{equation}
\begin{equation}
    \label{eq parameter q}
 \gamma, \delta \not=1,  \quad 
 \text{and either $\gamma \delta, \, \delta \geq 1 $ or $\gamma \delta, \, \delta \leq 1 $.}
\end{equation}
We also define 
$\lambda:=\frac{1-\gamma}{1-\frac{1}{\delta}} $ and $   q:=1-\frac{1}{\lambda}$. 

The Epstein-Zin aggregator $f$ and the bequest utility function $g$ are given by
\begin{equation}\label{eq aggregator}
f(C,v;  \delta,\gamma, \eta) 
:=\eta \lambda v\left(\left(\frac{C}{((1-\gamma) v)^{\frac{1}{1-\gamma}}}\right)^{1-\frac{1}{\delta}}-1\right)
\and
g(C;\gamma,\epsilon,\eta) = \frac{(\eta \epsilon)^{\lambda}}{1-\gamma} C^{1-\gamma},
\end{equation}
on the domain $C>0$ and $(1-\gamma)v>0$.

\subsection{\emph{N}-player Games with Epstein-Zin Preferences}
We next describe the game that $N \in \mathbb N$ investors with relative performance concerns and recursive utility play. 
Investors' preferences are characterized by the Epstein-Zin aggregators $$
f_i(C,v):= f (C,v; \delta_i, \gamma_i,\eta_i ) $$
   and bequest utility functions 
   $$
g_i(C,v):= g(C,v;\gamma_i,\epsilon_i,\eta_i),
$$
for $f$ and $g$ defined in \eqref{eq aggregator} with our standing assumption (\ref{eq parameter q}) on the parameters $\gamma_i, \delta_i$.

The investors have access to financial markets as in \cite{lacker.soret2020}. 
{Their initial wealth is given by square integrable independent positive random variables $x_0^1,...,x_0^N$ on a given complete probability space $(\Omega, \mathcal F, \mathbb P )$.
Investors' wealth is also affected by independent Brownian motions $B, W^1,...,W^N$, which are independent from $x_0^1,...,x_0^N$.}
Denote by $\mathbb F^N := (\mathcal F ^N_t)_{t \in [0,T]}$ the right-continuous extension of the filtration generated by $B, W^1,...,W^N$ and $x_0^1,...,x_0^N$, augmented by the $\mathbb P$-null sets. 

A (consumption--portfolio) strategy (or policy) $ \alpha = (c,\pi)$ is a couple of $ (0, \infty) \times \mathbb R$-valued $\mathbb F^N$-progressively measurable processes such that the boundedness conditions
\begin{equation}
    \label{eqN integrability of strategies NE}
    0 < \essinf c \leq \esssup c < \infty
    \and
    \esssup |\pi| < \infty,
\end{equation}
are satisfied. 
We denote by  $\mathcal A_N$ the set of  consumption-portfolio strategies.  
Policies will be denoted either by $\alpha_i = (c_i(t),\pi_i (t) )_{t \in [0,T]}$ or by $\alpha = (c_t,\pi_t )_{t \in [0,T]}$, depending on whether we want to refer to the specific investor $i$ or not.
A strategy $\alpha=(c,\pi)$ is  simple if 
\begin{equation}
\label{eq def simple strategy}
\text{ $c$ is a deterministic function and $\pi \in \R$ is a constant.}
\end{equation}
A strategy profile $\boldsymbol \alpha = (\alpha_1, ..., \alpha_N)$ is said to be simple if $\alpha _i$ is a simple strategy for any $i=1,...,N$.

For a policy $\alpha_i$, the  wealth process of investor $i$ is given by 
\begin{equation}
\label{eqN SDE players state}
dX^i_t= \pi_i(t) X_t^i ( \mu_i dt + \nu_i d W^i_t + \sigma_i dB_t ) - c_i(t) X^i_t dt, \quad X^i_0 = x^i_0
\end{equation}
for idiosyncratic volatility $\nu_i\ge 0$, common volatility  $\sigma_i\ge 0 $ with $\nu_i+\sigma_i>0$ and drift $\mu_i \in \mathbb R$.

The utility process $ (V_t^{i}{(\boldsymbol \alpha) })_{t \in [0,T]}$ of investor $i$    is given by the solution to the backward stochastic differential equation (BSDE, in short)
\begin{equation}
\label{eqN utility players}
 V_t^{i}{(\boldsymbol \alpha) }= \mathbb E \left[\left. \int_t^T f_i \left(c_i (s) X_s^i (\overline c_s \overline X _s )^{- \theta_i}, V_s^{i}{(\boldsymbol \alpha) }\right) ds + g_i \left(X_T^i \overline X _T^{-\theta_i}\right) \right| \mathcal F _t \right].
\end{equation}
The strategic interaction among players derives from the relative performance concerns modeled through the geometric averages of consumption and wealth  
$$ 
\overline X _t := \Big( \prod_{j=1}^N X_t^j \Big) ^{1/N} \quad \text{and} \quad \overline c _t := \Big( \prod_{j=i}^N c_j(t) \Big) ^{1/N}.
$$

\begin{remark}
    \label{remark BSDE well defined NE}
    \begin{enumerate}
    \item 
  For any  strategy profile $\boldsymbol \alpha$, the utility process defined by the BSDE  \eqref{eqN utility players} is well-defined by Theorem 3.1 in \cite{seiferling.seifried.2016} (see also \cite{kraft.Seiferling.Seifried.2017optimal}  for stochastic differential utilities in a context similar to ours in which $c$ is not necessarily continuous). 
   Indeed, Theorem 3.1 in \cite{seiferling.seifried.2016} yields    existence and uniqueness  of $V^i(\boldsymbol \alpha )$ provided that 
   the process
   $ P_t^{\boldsymbol \alpha} : = c_i(t) X_t^i (\overline c_t \overline X _t )^{- \theta}$ is positive and   satisfies
   $$
 \mathbb E \Big[ \int_0^T  |P_t^{\boldsymbol \alpha}|^k   dt  + |P_T^{\boldsymbol \alpha}|^k \Big] < \infty, \quad \text{for any $k \in \mathbb R$.}
  $$
  The latter integrability easily follows from the condition \eqref{eqN integrability of strategies NE} in the definition of strategies $\alpha_i$.
\item 
In the case of $\gamma_i = \frac{1}{\delta_i}$, we recover the case of time--additive intertemporal preferences. Indeed, 
using Itô's lemma, one can show that the ordinally equivalent utility process
$$ U^i_t := \frac{1}{\eta_i} \left(V_t^{i}{(\boldsymbol \alpha) } \right)^{\frac{1}{\lambda}}$$ 
satisfies
the equation
$$ U^i_t =  \mathbb E \left[ \int_t^T   e^{-\eta_i (s-t)} \frac{ (c_i (s) X_s^i (\overline c_s \overline X _s )^{- \theta_i})^{1-\gamma_i} }{1-\gamma_i} ds +      \epsilon_i \frac{( X_T^i  \overline X _T ^{- \theta_i})^{1-\gamma_i}}{1-\gamma_i}  \right].
$$
This parametrization of investor i's preferences shows that    $\epsilon_i$ is indeed to be interpreted as the weight of the bequest utility, justifying our choice of the factor $ (\eta \epsilon)^{\lambda}$ in Equation\footnote{For an extensive discussion of the bequest motive in recursive preferences, compare \cite{kraft2022bequest}.}  (\ref{eq aggregator}).

In particular, we   cover the time-additive case for non-zero discount rates that was not solved for in  \cite{lacker.soret2020}, where both the safe interest rate and the discount rate are zero. Such an assumption does not come   without loss of generality.  Thus, it is economically important to study   the case of distinct interest and discount rate. 
\end{enumerate}
\end{remark}

Given   a strategy profile $\boldsymbol{\alpha} = (\alpha_i )_{i=1, \ldots, N} \in \mathcal A_N^N$, denote by $\boldsymbol{\alpha} _{-i} := (\alpha_j )_{j \neq i}$ the vector of strategies chosen by investors $j \ne i$  and write, {as usual in game theory},  $\boldsymbol \alpha = (\alpha_i, \boldsymbol{\alpha} _{-i}) : = (\alpha_1, ..., \alpha_N) \in \mathcal A _N^N$. 
A strategy profile $ \hat{\boldsymbol \alpha}$ is a Nash equilibrium (NE, in short)   if
\begin{equation}
    \label{eqN def NE}
    \text{
    $V_0^i (\hat{\boldsymbol \alpha}) \geq V_0^i (\alpha_i, \hat{\boldsymbol \alpha} _{-i}),$
     for any  $\alpha_i \in \mathcal A _N$ and $i=1,...,N$.}
\end{equation}

\begin{theorem}\label{thm existence NE}
There exists a unique Nash equilibrium in simple strategies  $\hat{\boldsymbol{\alpha}}_N = (\hat \alpha_{1,N},..., \hat \alpha_{N,N})$   given by
\begin{align} 
   \label{PortfolioFiniteN}
   \hat \pi _{i,N}  & = {\frac{1}{\gamma_i }} \frac{  \mu_i}{ (1 + \frac{\theta_i}{N} (\frac{1}{\gamma_i} - 1))\nu_i^2+  \sigma_i^2} - {\sigma_i} \theta_i \big( \frac{1}{\gamma_i} -1 \big) \frac{   \Pi_N }{(1 + \frac{\theta_i}{N} (\frac{1}{\gamma_i} - 1))\nu_i^2+  \sigma_i^2},\\
   \hat c _{i,N} (t) &= 
  \begin{cases} 
      \Big( \frac{1}{{\chi_2^{i,N}}} + \big(  \frac{1}{\chi_1^{i,N}} - \frac{1}{\chi_2^{i,N}} \big)e^{ - \chi_2^{i,N}(T-t)}   \Big)^{-1} & \text{ if }\chi_2^{i,N}\neq 0, \\
      \big(T-t+\frac{1}{\chi_1^{i,N}}\big)^{-1} & \text{ if } \chi_2^{i,N}= 0,
   \end{cases}  
\end{align}
where
\begin{align*}
   \Pi_N & = \frac{E_N}{1+F_N},\\
    E_N &= \frac 1 N
    {\sum_{j=1}^N\frac{   \sigma_j \mu_j}{ {\gamma_j}{ (1 + \frac{\theta_j}{N} (\frac{1}{\gamma_j} - 1))\nu_j^2+  
 \gamma_j \sigma_j^2}}}, \\ 
    F_N &={\frac1N \sum_{j=1}^N \frac{ {\sigma_j^2} \theta_j (1-\gamma_j)}{{\gamma_j}{ (1 + \frac{\theta_j}{N} (\frac{1}{\gamma_j} - 1))\nu_j^2 +\gamma_j  \sigma_j^2}}} , \\
 \chi_1^{i,N} &
    =   \epsilon_i^{-\delta _i } \Big( \prod_{j=1}^N \epsilon_j^{- \delta _j } \Big) ^{\frac{\theta_i (1-\delta_i)}{N- \sum_{j=1}^N \theta _j (1-\delta_j)} }, \\
\chi_2^{i,N} &
    =  - \frac{\delta_i}{\lambda_i} \rho_{i,N} - \theta_i (1-\delta_i) \frac{ \sum_{j=1}^N \frac{\delta_j}{\lambda_j} \rho_{j,N}}{N - \sum_{j=1}^N \theta _j (1-\delta_j)}, \\
\rho_{i,N} &= - \eta_i \lambda_i + (1-\gamma_i) \Big[  - \theta_i  \hat \mu _i  + \frac{1}{2} \theta_i (1 + \theta_i (1-\gamma_i)) (\hat \nu _i^2+ \hat \sigma _i^2) \\  
    & \quad \quad \quad   \quad \quad  \quad \quad  \quad \quad    - \frac{1}{2}  (1-\frac{\theta_i}{N}) \frac{( \sigma_i \hat \sigma _i \theta_i (1-\gamma_i) - \mu_i )^2}{ (  (1-\frac{\theta_i}{N}) (1-\gamma_i ) -1)(\nu_i^2+ \sigma_i^2)} \Big], \\
\hat \nu _i &= \frac1N \sqrt{ \sum_{j \ne i} ( \nu _j \hat \pi_{j,N})^2}, \quad
\hat \sigma _i = \frac1N \sum_{j \ne i}  \sigma_j \hat \pi_{j,N} , \\ \notag
\hat \mu _i &=\frac1N \sum_{j\ne i} \big(  \mu_j \hat \pi_{j,N}  - \frac12 \hat \pi_{j,N}^2 ( \nu_j^2 + \sigma_j^2) \big) + \frac12 ( \hat \nu _i^2 + \hat \sigma _i^2).
\end{align*}
\end{theorem}

\begin{corollary}
   In  the case of a single common stock,   i.e.   
    $ \nu_i=0$ for all $i$, a common drift  $\mu_i=\mu$ and common volatility $\sigma_i=\sigma$, the Nash equilibrium strategies simplify to 
\begin{align*} 
\hat \pi _{i,N}  & = \frac{1}{\gamma_i} \frac{  \mu}{\sigma^2} -  \theta_i \left( \frac{1}{\gamma_i} -1 \right) \frac{ \mu}{\sigma^2} \frac{ \frac{1}{N} \sum_{j=1}^N \frac{1}{\gamma_j}}{1+\frac{1}{N} \sum_{j=1}^N \theta_j \left( \frac{1}{\gamma_j}-1\right) },\\
   \hat c _{i,N} (t) &= 
  \begin{cases} 
      \left(  \frac{1}{{\chi_2^{i,N}}} + \big(  \frac{1}{\chi_1^{i,N}} - \frac{1}{\chi_2^{i,N}} \big)e^{-\chi_2^{i,N}(T-t)} \right)^{-1} & \text{ if }\chi_2^{i,N}\neq 0, \\
      \left(T-t+\frac{1}{\chi_1^{i,N}}\right)^{-1} & \text{ if } \chi_2^{i,N}= 0,
   \end{cases}  
\end{align*}
\end{corollary}

\subsection{Mean-Field Games with Epstein-Zin Preferences}
We now introduce the  mean-field game following the notation of \cite{lacker.soret2020, lacker2019mean}.
Let the probability space $(\Omega,\mathcal{F},P)$ be endowed with a filtration $\mathbb F  = (\mathcal F _t) _{t \in [0,T]}$ satisfying the usual conditions.
Let $B$ and $W$ be independent $\mathbb F$-Brownian motions, modeling common and idiosyncratic noise, resp.
We underline that the $\sigma$-algebra $\mathcal F_0$ is independent from $B$ and $W$, but it is not necessarily trivial.

The representative investor is characterized by the initial wealth $x_0$,  the drift and volatility parameters of her stock $\mu, \nu, \sigma$, and the preference parameters $\gamma,\delta,  \epsilon, \eta,\theta$.\footnote{The master's thesis \cite{wang2023} derives the equilibrium of the same model in the special case in which the  preference parameters are deterministic.} 
We call  
$$
\mathcal{I}:=(0, \infty) \times \R \times [0, \infty) \times [0, \infty) \times (0, \infty) \times(0, \infty)  \times(0, \infty) \times [0,1]
$$
the corresponding \emph{type space} of our game with typical element
$ (x_0, \mu, \nu, \sigma, \eta, \gamma, \delta, \epsilon, \theta)$.
The type of representative investor is described by  a $\mathcal F_0$-measurable  random variable $\mathcal{T}: \Omega\rightarrow \mathcal{I}$, and we assume  
{$x_0$ to be square integrable and}
\begin{equation}
    \label{assumption parameters MFG}
 \sigma+\nu>0, \quad \gamma, \delta \not=1,  \quad 
 \text{and either $\gamma \delta, \, \delta \geq 1 $ or $\gamma \delta, \, \delta \leq 1 $, \quad $\mathbb P$-a.s.}
\end{equation}

The Epstein-Zin aggregator $f$ and terminal cost $g$ of the representative investor are respectively given by the (random) functions
$$
f:= f (\cdot;  \delta,\gamma, \eta) 
\quad \text{and}\quad
g:= g(\cdot;\gamma,\epsilon,\eta),
$$
for $f$ and $g$ defined in \eqref{eq aggregator}.

A (consumption-investment) policy $ \alpha = (c,\pi)$ is a couple of $ (0, \infty) \times \mathbb R$-valued $\mathbb F$-progressively measurable processes such that the boundedness conditions \eqref{eqN integrability of strategies NE} are satisfied. 
We denote by  $\mathcal A$ the set of  policies.
A strategy $\alpha=(c,\pi)$ is said to be simple if 
\begin{equation}
\label{eq def simple strategy MFG}
\text{ $c$ is $\mathcal F _0$-measurable and $\pi$ is $\mathcal F _0$-measurable constant in time.}
\end{equation}
For a policy $\alpha$, the  wealth process of the representative investor is given by 
\begin{equation}\label{eq MFG SDE players state}
dX_t= \pi_t X_t ( \mu dt + \nu d W_t + \sigma dB_t ) - c_t X_t dt, \quad X_0 = x_0.
\end{equation}

Let $\mathbb F^B = (\mathcal F^B_t )_{t \in [0,T]}$ denote the natural filtration of the Brownian motion $B$. 
We denote by $Y$ and $m$ the generic geometric mean wealth and geometric mean consumption rate of the population of investors, respectively.
The processes $Y$ and $m$ are assumed to be $\mathbb F^B$-progressively measurable.
The representative investor takes $Y$ and $m$ as given, and aims
at maximizing, over the policy $\alpha$, her utility $V=V(\alpha; m,Y)$ which is given by the solution to the BSDE
\begin{equation}
    \label{DefSDU}
 V_t= \mathbb E \left[ \left.\int_t^T f \left(c_s X_s \left(m_s Y_s \right)^{- \theta}, V_s \right)  ds +  g \left(X_T Y_T^{-\theta}\right) \right| \mathcal F _t \right].
\end{equation}

\begin{remark}
    \label{remark BSDE well defined MFG}
    The existence of stochastic differential utility $V$ solving (\ref{DefSDU}) with random initial parameters can be shown along the lines of \cite{seiferling.seifried.2016} (see also Remark \ref{remark BSDE well defined NE}). In fact, existence of SDU when   initial wealth $x_0$,    drift $\mu$ and $\nu$ and volatility $\sigma$ are random $\mathcal{F}_0$-measurable variables is already covered by Theorem 3.1 in \cite{seiferling.seifried.2016}. The authors do not discuss the case of random initial preference parameters $(\delta, \gamma, \epsilon, \theta, \eta)$, though,  yet an inspection of the proofs shows that the arguments go through.
    
    Using a martingale representation argument, one can   show that
    the process $V$ solving  (\ref{DefSDU}) satisfies the backward stochastic differential equation
       $$dV_t = - f ( c_t X_t (m_t Y_t )^{- \theta}, V_t )  dt
       + Z^1_t dW_t + Z^2_t dB_t, \quad V_T = g  ( X_T Y_T^{-\theta} ),
       $$
       for some square-integrable progressively measurable processes $Z^i,i=1,2$.
\end{remark}

In equilibrium, the assumed geometric mean consumption rate $m$ has to be equal to the geometric mean consumption rate of the population that is given by $$\exp (\mathbb E [ \log c_t | \mathcal F^B_t] )$$ and the assumed geometric mean wealth has to equal the population geometric mean wealth $$\exp (\mathbb E [ \log Y_t | \mathcal F^B_t ] ).$$ 
We refer to \cite{lacker.soret2020, lacker2019mean} for more details.

\begin{definition}\label{mfgdef}
Let $\hat \alpha=(\hat c, \hat \pi)$ be a policy with corresponding wealth process $\hat X$. 
Let 
$$
\hat Y_t =\exp(\mathbb E [ \log \hat X_t | \mathcal F^B_t] )
\and
\hat m_t = \exp(\mathbb E [ \log \hat c_t | \mathcal F^B_t] ).
$$
$\hat \alpha$ is a mean-field game equilibrium (MFGE) if $\hat \alpha$ maximizes the recursive utility $V(\cdot; \hat m,\hat {Y})$.
\end{definition}

\begin{theorem}\label{thmfg}
There is a unique   MFGE in simple strategies  $\hat \alpha =(\hat c, \hat \pi)$   given by
\begin{align*}
\hat \pi  &  = \frac{\mu}{\gamma (\sigma^2+\nu^2)}- \theta \left(\frac{1}{\gamma}-1\right) \frac{\sigma}{(\sigma^2+\nu^2)}\frac{E}{1+F},\\
\hat c _t &= 
  \begin{cases} 
      \Big( \frac{1}{{\chi_2}} +  \big( \frac{1}{\chi_1} - \frac{1}{\chi_2} \big)e^{-\chi_2(T-t)} \Big)^{-1} & \text{ if }\chi_2\neq 0, \\
      \big(T-t+\frac{1}{\chi_1}\big)^{-1} & \text{ if } \chi_2= 0,
   \end{cases}  
\end{align*}
where 
\begin{align*}
E &= \begin{matrix}{\mathbb{E}\left[\frac{\mu \sigma}{\gamma(\sigma^2+\nu^2)}\right]}\end{matrix}, \\
F &= \begin{matrix} \mathbb{E}\left[ \theta\left(\frac{1}{\gamma}-1\right)\frac{\sigma^2}{\sigma^2+\nu^2} \right] \end{matrix}, \\
\chi_1 & =  \epsilon^{- \delta  } \exp \Big( \frac{\theta (1-\delta)}{1+ \mathbb E [\theta (\delta-1) ]} \mathbb E [ -\delta \log \epsilon ] \Big), \\
\chi_2  &=   \theta (\delta-1)  \frac{\mathbb E \left[ \frac{\delta}{\lambda} \rho\right] }{1+\mathbb E \left[ \theta(\delta-1)\right] } - \frac{\delta}{\lambda} \rho, \\
\rho &:= - \eta \lambda + (1-\gamma) \Big[  - \theta  \hat \mu   + \frac{1}{2} \theta (1 + \theta (1-\gamma))  \hat \sigma ^2   + \frac{1}{2} \frac{( \sigma \hat \sigma  \theta (1-\gamma) - \mu )^2}{ \gamma(\nu^2+ \sigma^2)} \Big],\\
\hat \sigma &:= \mathbb E [\hat \pi \sigma ], \quad
\hat \mu := \mathbb E [ \hat \pi \mu]  - \frac12 \big( \mathbb E [\hat \pi^2 ( \nu^2 + \sigma^2) ] - (\mathbb E [\hat \pi \sigma ] )^2 \big). 
\end{align*}
\end{theorem}

\subsection{The Economics of Relative Performance Concerns with Recursive Preferences}
\label{econdiscuss}
In the following, we discuss the new economic features of strategic behavior in equilibrium when relative performance concerns and recursive preferences matter. We focus mainly on the mean-field game, but emphasize the differences to finite player games in passing.

Let us start with the optimal portfolio choice. 
The optimal investment rule consists of the usual Samuelson-Merton term $\frac{\mu}{\gamma (\sigma^2+v^2)}$ and a correction term for the relative performance concerns. Recursive utility allows to disentangle risk aversion $\gamma$ and elasticity of intertemporal substitution $\delta$. Note that the investment decision is affected by risk aversion, not by the elasticity of intertemporal substitution, a finding that reemphasizes the previous results  in  the literature on recursive preferences.  Otherwise, the investment policy coincides with the investment policy of a time-additive investor with performance concerns in \cite{lacker.soret2020}. 

It is noteworthy to observe the scenarios in which the pure Merton portfolio emerges. This occurs when the investor disregards competitors, i.e. $\theta=0$. The Merton portfolio is also obtained in the case where $\gamma=1$, indicating an investor with unit relative risk aversion. Moreover, competition's impact on investment diminishes when markets are entirely separate and independent ($\sigma=0$). This latter phenomenon is a mean-field effect. For finite $N$, competition influences portfolio selection, albeit diminishing with increasing $N$, as demonstrated in Equation (\ref{PortfolioFiniteN}).

 In order to discuss the comparative statics of portfolio choice, we now consider the case of a common market without idiosyncratic noise.
\begin{corollary}\label{corollarykey}
Assume that $(\mu, \nu, \sigma)$ is deterministic with $\nu=0$ and $\mu, \sigma>0$. Then the optimal investment simplifies to $$\pi_\gamma=\frac{\mu}{\sigma^2}\left(\frac{1}{\gamma}+\frac{\mathbb{E}[\frac{1}{\gamma}]\theta(1-\frac{1}{\gamma})}{(1+\mathbb{E}(\theta(\frac{1}{\gamma}-1)))}\right)$$
\end{corollary}
%\fr{Interpretation, Comparison to Lacker, why is recursive utility intreresting, pictures: consumption...}

 We have that
 $$\diffp{\pi_\gamma}{\gamma}=\frac{\mu}{\sigma^2}\left(\frac{\theta\frac{\mathbb{E}[\frac{1}{\gamma}]}{1+\mathbb{E}(\theta(\frac{1}{\gamma}-1))}-1}{\gamma^2}\right),$$
so that if we set
$$\theta^*:=\frac{1+\mathbb{E}(\theta(\frac{1}{\gamma}-1))}{\mathbb{E}\left[\frac{1}{\gamma}\right]},$$ then whenever $$\frac{\theta}{\theta^*}> 1,$$
more risk aversion increases the level of investment in the portfolio, i.e. $\diffp{\pi_\gamma}{\gamma}>0$, while if
$$\frac{\theta}{\theta^*}< 1,$$
more risk aversion decreases the level of investment in the portfolio, that is, $\diffp{\pi_\gamma}{\gamma}<0$.
 Hence, as in Lacker and Soret \cite{lacker.soret2020}, one can have  competitive agents (high $\theta$) with high level of risk aversion (high $\gamma$) that behave like noncompetitive agents (low $\theta$) with  low risk aversion (low $\gamma$). However, recursive utility allows us to exclude the effect of intertemporal elasticity of substitution.

Now let us turn to the equilibrium consumption policies.
Observe that the terminal level of consumption,  \( \chi_1 \),  instead, is independent of risk aversion, but it does depend on the level of intertemporal elasticity of consumption \( \delta \). It is also independent of the discount rate $\eta$.
If we set the weight of bequest utility to $1$, we even get $\chi_1=1$ throughout.

  For large horizon, the consumption rate is essentially given by the constant $\chi_2$. 
If we compare $\chi_2$ with the corresponding constant $\beta$ in \cite{lacker.soret2020}, Equation (36), we see that recursivity leads to an additional   parameter $\lambda$ which is 1 in the time-additive case.  In the empirically reasonable case of large risk aversion and large intertemporal elasticity of substitution, $\lambda$ is not $1$, yet negative, so the effect on $\chi_2$ is quite remarkable.
The equilibrium behavior of consumption is influenced by both the elasticity of intertemporal substitution and the level of risk aversion. We show here that by implicitly assuming that $\lambda=1$, or equivalently that $\gamma=\frac{1}{\delta}$, the model in Lacker and Soret  \cite{lacker.soret2020} may lead to misleading conclusions about equilibrium consumption behavior. In contrast, since in our setting $\lambda$  can be different from unity, we are able to avoid such an issue.  

In order to examine   the behavior of the equilibrium function $c(t)$, fix the level of risk aversion to a deterministic constant $\gamma$ and  set $\eta=\epsilon=1$, so $\chi_1=1$.    Immediate calculations show that when $\chi_2\neq 0$ consumption is increasing over time whenever $\chi_2< \chi_1=1$, constant when $\chi_2=\chi_1$, and decreasing over time when $\chi_{2}>\chi_1=1$.

%\fr{I propose to argue as follows in the discussion below. }

Now observe that if we set 
$$\iota:=\Big[  - \theta  \hat \mu   + \frac{1}{2} \theta (1 + \theta (1-\gamma))  \hat \sigma ^2   + \frac{1}{2} \frac{( \sigma \hat \sigma  \theta (1-\gamma) - \mu )^2}{ \gamma(\nu^2+ \sigma^2)} \Big],$$
then it holds  that
\begin{align}
\chi_2 &= \theta (\delta-1)  \frac{\mathbb E \left[ \frac{\delta}{\lambda} \rho\right] }{1+\mathbb E \left[ \theta(\delta-1)\right] } - \frac{\delta}{\lambda} \rho \nonumber \\
&=  \left[\theta\left(\frac{{\delta-1}}{1-\gamma}\right)\frac{\mathbb{E}\left[{\rho}{(\delta-1)}\right]}{1+\mathbb{E}\left[\theta(\delta-1)\right]}-\frac{\delta-1}{1-\gamma}\rho\right] \nonumber \\
&= \frac{1}{1-\gamma}\left[\theta(\delta-1)\frac{\mathbb{E}\left[\rho{(\delta-1)}\right]}{1+\mathbb{E}\left[\theta(\delta-1)\right]}-{(\delta-1)}\rho\right] \nonumber \\
&= \frac{1}{1-\gamma}\left[\theta(\delta-1)\frac{\mathbb{E}\left[\rho{(\delta-1)}\right]}{1+\mathbb{E}\left[\theta(\delta-1)\right]}+\delta(1-\gamma)-(\delta-1)(1-\gamma)\iota\right].\label{avemaria}
\end{align}

%\fr{In equation (14), there is a $\lambda$ missing. I neither understand how $\lambda$ can disappear in equation (15) and afterwards.}
Assume that 
\begin{equation}\label{gratiaplena}
\theta\frac{\mathbb{E}\left[\rho{(\delta-1)}\right]}{1+\mathbb{E}\left[\theta(\delta-1)\right]}+(1-\gamma)-(1-\gamma)\iota>0,
\end{equation}
then there are two main cases to consider:
\begin{enumerate}
 \item  When $\gamma>1$, then there exists $\delta^*\in\mathbb{R}$ such that for every $\delta\in (\delta^*,\infty)$, $c(t)$ is strictly increasing, and  strictly decreasing if $\delta\in (-\infty,\delta^*)$.\footnote{Indeed, if $\gamma>1$ the term in \eqref{avemaria} is decreasing in $\delta$.  It follows that  for some $\delta^*$, $\chi_2<\chi_1=1$ for every $\delta\in (\delta^*,\infty)$. The case $\gamma<1$ is symmetric.} In words, if the consumer is sufficiently elastic when compared to an average level of elasticity of the population of consumers,  the consumer will display increasing consumption over time, which reflects the   willingness to delay a higher level of consumption in the future. % which is a common assumption in asset pricing (see for example \cite{bansal2004risks}), the analysis is similar to \cite{lacker.soret2020}. In the main case of interest, there exists  a non-empty interval $[\underline{\gamma},\bar\gamma]$ such that  consumption is increasing for $\gamma\notin (\underline{\gamma},\bar\gamma)$, decreasing for $\gamma \in  (\underline{\gamma},\bar\gamma)$ and constant for $\gamma\in\{\underline{\gamma},\bar\gamma\}$.
 \begin{figure}[h!] 
    \centering
    \begin{subfigure}[b]{0.8\textwidth}
      \includegraphics[width=\textwidth]{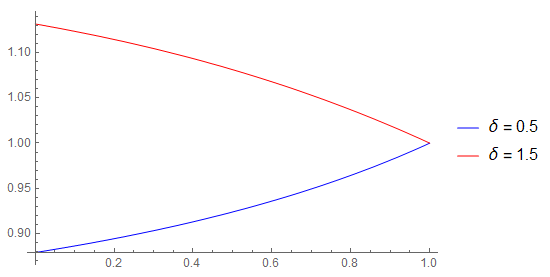}
        \caption{$\gamma=2$}
    \end{subfigure}
    % \hfill
    %\begin{subfigure}[b]{0.49\textwidth}
    %   \includegraphics[width=\textwidth]{delta_2.png}
   %     \caption{$\gamma=0.8$}
   % \end{subfigure}
    \caption{Graphs displaying the function $c(t)$   for different values of the parameters. In this case we have 
$\frac{\mathbb{E}[\rho(\delta-1)]}{1+\mathbb{E}[\theta(\delta-1)]}
=0.4$, $\iota=0.2$.}\label{figure1}
\end{figure}
\item The case  $\gamma<1$ is symmetric:  there exists $\delta^*\in\mathbb{R}$ such that for every $\delta\in (-\infty,\delta^*)$, $c(t)$ is increasing, and decreasing if $\delta\in (\delta^*,\infty)$.% (a different plausible value, see 
% for example \cite{hansen1983stochastic}), the previous result is reversed. Specifically, there exists a non-empty interval of risk aversion coefficients $[\underline{\gamma},\bar\gamma]$, where consumption behavior changes based on the level of risk aversion. For a level of risk aversion $\gamma\in (\underline{\gamma},\bar\gamma)$, consumption is increasing. However, for $\gamma\not\in (\underline{\gamma},\bar\gamma)$, consumption is decreasing. Consumption remains constant at the boundaries, that is, for $\gamma\in\{\underline{\gamma},\bar\gamma\}$.
\end{enumerate}

\vspace{0.20cm}

\noindent
Now observe that in the time-additive case of Lacker and Soret \cite{lacker.soret2020}, $\lambda=1$. A typical empirical value for risk aversion is $\gamma=2$ (see the discussion in \cite[p.~574]{ju2012ambiguity}). It follows that  implicitly in the time-additive model one assumes $\delta=1/2$. This might be very misleading as   $\delta>1$ is a common assumption in asset pricing (e.g., see  \cite{bansal2004risks}).  If we take $\delta>1$, it follows that $\lambda$ is negative,  thus changing the shape of equilibrium consumption. In particular,  Figure \ref{figure1} illustrates this point: the function $c(t)$ can switch from increasing to decreasing depending on whether $\delta$ is smaller or greater than unity. Hence, under the framework of stochastic differential utility, it is possible to distinguish the effect of the elasticity of intertemporal substitution (EIS) from that of risk aversion on equilibrium consumption behavior, avoiding potentially misleading conclusions.

%To make a comparison with to the case of standard intertemporal utility, observe that if we assume that 
%$$\frac{\mathbb{E}\left[\frac{1}{\gamma}\right]}{(1+\mathbb{E}(\theta(\frac{1}{\gamma}-1)))}=\frac{\mathbb{E}\left[\delta\right]}{(1+\mathbb{E}(\theta(\delta-1)))},$$
%we have that $\gamma\geq \frac{1}{\delta}$ implies a higher degree of investment when $\theta> \theta^*$ and a lower level of investment when $\theta< \theta^*$ compared to the case of standard utility (i.e. $\gamma=\frac{1}{\delta}$). The exact opposites holds whenever $\gamma\leq \frac{1}{\delta}$.

%a 2-player game with $(x_0^i, \mu_i, \nu_i, \sigma_i, \eta_i, \gamma_i,\varepsilon_i, \theta_i)  =(x_0, \mu, \nu, \sigma, \eta, \gamma, \varepsilon, \theta), $ for $i =1,2,$ and $\delta_1 \ne \delta_2$.  Further assume that $\chi_1^{1,2}=\chi_1^{2,2}=1$.
%In this case, we have $\hat \pi_{1,2} = \hat \pi _{2,2}$, thus the investment choices coincide. Immediate calculations show that
%$$\chi_2^{1,N} = \frac{\delta_1-1}{1-\gamma} \left(\theta \frac{ \frac{1}{2}\left(\delta_1 \rho_{1,N}+\delta_2 \rho_{2,N}\right)}{1 -  \frac{1}{2}\left(\theta _1 (1-\delta_1)+\theta_2(1-\delta_2)\right)}-\rho_{1,N}\right) $$
%$$.$$
%and
%$$\chi_2^{2,N} = \frac{\delta_2-1}{1-\gamma} \left(\theta \frac{ \frac{1}{2}\left(\delta_1 \rho_{1,N}+\delta_2 \rho_{2,N}\right)}{1 -  \frac{1}{2}\left(\theta _1 (1-\delta_1)+\theta_2(1-\delta_2)\right)}-\rho_{2,N}\right) ,$$
%$$a.$$

\subsection{The Mean-Field Game as a Limit of Finite Player Games }
In this subsection we discuss the relations between the MFG and the $N$-player game for large $N$.
In order to simplify the discussion, we assume (with no further reference) 
that
\begin{equation}\label{eq assumption iid}
\text{$x_0^i$ are i.i.d.\ with the same distribution as $x_0$, }
\end{equation}
and
\begin{equation*}\label{eq assumption symmetric game}
(\mu_i, \nu_i, \sigma_i, \eta_i, \gamma_i, \delta_i, \epsilon_i, \theta_i) = (  \mu, \nu, \sigma, \eta, \gamma, \delta, \epsilon, \theta) , 
 \text{ for any  $i=1,...,N$}.    
\end{equation*}
In particular, notice that we are assuming the coefficients of the MFG to be deterministic.

We first state the following convergence result. 
\begin{theorem}[Convergence to the MFGE]
    \label{thm convergence}
    Let $\hat {\boldsymbol{\alpha}} _N = (\hat c_{i,N}, \hat \pi_{i,N})_{i=1,...,N}$ be the simple NE equilibrium of the $N$-player game as in Theorem \ref{thm existence NE} and let $(\hat c, \hat \pi)$ be the simple MFGE  as in Theorem \ref{thmfg}, with corresponding  $(\hat m , \hat Y)$. 
    We have convergence of the equilibrium strategies
    \begin{align*}
        \sup_{t \in [0,T]} | \hat c _{i,N}(t)  - \hat c (t)| + | \hat \pi _{i,N}  - \hat \pi| = O(1/N), \text{ for any $i =1,...,N$,}
    \end{align*}
    of the mean-field terms
    \begin{align*}
        & \sup_{t \in [0,T]} |  \overline{X}_t^N -  \hat Y _t| = O\big(N^{-\frac12} (\log N)^{-\frac12 -a} \big),  \text{ $\mathbb P$-a.s.,} \\
        & \sup_{t \in [0,T]} | \overline{c}^N(t)  - \hat m (t)| = O(1/N),
    \end{align*}
    for any $a >0$, and of the extepcted values at equilibrium
    $$
       { | \mathbb E [  \log V^i_0 (\hat{\boldsymbol{\alpha}}_N ) - \log V_0 (\hat \alpha , \hat m, \hat Y) ]  | = O(1/N), }
    $$
    for any $i$.
\end{theorem}

We next show that the MFGE does indeed approximate NE of the corresponding game with finitely many players.
In particular, we will use the MFGE  $\hat \alpha=(\hat c, \hat \pi)$ in order to construct approximate NE for the original $N$-player game.
For any $N \in \mathbb N _0$, define the strategy profile $\hat{\boldsymbol{ \alpha}}_N^\infty$ in which all players are using the simple (deterministic) strategy $\hat \alpha$; that is, 
\begin{equation*}
    \hat{\boldsymbol \alpha}_N^\infty := (\hat \alpha, ...,\hat \alpha ) \in \mathcal A_N^N. 
\end{equation*}
We have the following result. 
\begin{theorem}[Approximate NE]
    \label{theorem approximation}
The  strategy profile $\hat{\boldsymbol \alpha}_N^\infty$ is an approximate NE of order $O(1/N)$ as $N \to \infty$; 
that is, 
$$
\big| V_0^i (\hat{\boldsymbol \alpha}_N^\infty) -\sup _{\alpha_i \in \mathcal A _N} V_0^i (\alpha _i , \hat{\boldsymbol \alpha} _{-i, N}^\infty) \big| = O(1/N), \quad \mathbb P \text{-a.s.,}
$$
for any $ i=1,...,N$.
\end{theorem}

\section{Preliminary results on geometric-Bernoulli BSDEs}
\label{SecBernoulliBSDE}
Observe that, given a  profile $\boldsymbol \alpha$, for any $i=1,...,N$ the processes $X^i, \, \overline X$ are {generalized} geometric Brownian motions, hence so  it is $X^i (\overline X)^{-\theta}$.
Therefore, the system \eqref{eqN SDE players state}-\eqref{eqN utility players}, is a forward backward stochastic differential equation in which the forward component is a geometric Brownian motion and in which the backward component has Bernoulli driver (i.e.,  $f(C,v) = f_1 v + f_2(C) v^q$, for $q\geq 0$, $q \ne 0,1$).
Since the forward components do not depend on the backward component, we will refer to such type of systems as \emph{geometric-Bernoulli BSDEs}. 

An essential tool in the proof of our main results are  explicit solvability,  stability and  characterization of optimal controls when optimizing this type of systems.
We devote this section to address this preliminary results.

\subsection{Geometric-Bernoulli BSDEs} 
Let $(\tilde \Omega, \tilde{\mathcal{F}}, \tilde{\mathbb F}  = (\tilde{\mathcal F} _t) _{t \in [0,T]}, \tilde{\mathbb P})$ be a generic filtered probability space   satisfying the usual conditions, on which are defined independent $\tilde{\mathbb F}$-Brownian motions $\tilde W^1, \tilde W^2, B$.
Let $(x_0, \mu, \nu, \sigma, \eta, \gamma, \delta, \epsilon, \theta)$ be an $\tilde{\mathcal F}_0$-measurable $\mathcal I$-valued random variable  satisfying \eqref{assumption parameters MFG} $\tilde{\mathbb P}$-a.s., 
and take deterministic functions $b^2, \hat b ^2 :[0,T] \to [0,\infty)$ and  parameters $\mu^2 \in \R, \ \nu^2,\sigma^2 \geq 0$.
{Consider a positive random variable $y_0$ independent from $x_0$. The random variables $x_0$ and $y_0$ are assumed to be square integrable.} 
For  $\gamma, \delta \ne 1$, $\lambda, q$ as in \eqref{assumption parameters MFG},  $p, \theta \in [0,1]$,
let the (random) aggregator $f := f(\cdot; \eta, \delta , \gamma)$ and the terminal cost $g:= g (\cdot; \gamma, \epsilon)$ be as in \eqref{eq aggregator}.

A strategy $ \alpha = (c,\pi)$ is a couple of $ (0, \infty) \times \mathbb R$-valued $\tilde {\mathbb F}$-progressively measurable processes such the boundedness conditions in \eqref{eqN integrability of strategies NE} are satisfied. 
The space of strategies is denoted by $\tilde{\mathcal A}$.
Simple strategies such that $(c,\pi)$ is $\tilde{\mathcal F}_0$-measurable, with $\pi$ constant in time.
Let $\alpha  = (c,\pi)$ be a strategy, and consider the related solution $(X^\alpha,Y,V^\alpha, Z^\alpha)$ of the geometric-Bernoulli BSDE 
\begin{align}\label{eq geometric Bernoulli BSDE X Y V}
d X^\alpha_t &= \pi_t X^\alpha_t (\mu^1 +  \nu^1 d\tilde W^1_t + \sigma^1 d B _t) -c_t X_t dt, \quad X^\alpha_0 = x_0, \\ \notag
d Y_t &= Y_t ((\mu ^2- \hat b _t^2) dt + \nu^2 d \tilde W^2_t + \sigma^2 d  B _t), \quad Y_0 = y_0,\\ \notag
d V_t^\alpha &= - f( (c_t X_t)^p (b_t^2 Y_t )^{-\theta}, V_t^\alpha ) dt + Z^\alpha_t d ( \tilde W^1, \tilde W ^2, B )_t,  \quad  V_T^\alpha= g ((X_T)^p Y _T^{-\theta}). \notag 
\end{align}

For simple strategies, the solution of geometric-Bernoulli BSDEs is related to the solution of the Bernoulli ordinary differential equation (ODE, in short) 
$$
h' + \varphi (t) h + \psi (t) h^{1 - a } =0, \quad h (T)=1,
$$
for a suitable parameter $a$ and suitable continuous functions $\varphi, \psi : [0,T] \to \mathbb R$.
When $a \ne 1,0$, it is well known that the unique solution of the Bernoulli ODE is given by
\begin{equation}
    \label{eq SOL Bernoulli generic}
     h_{a, \varphi, \psi}(t) := \Big( e^{a \int_t^T \varphi(r) dr } + a \int_t^T \psi(r) e^{a \int_t^s \varphi (r) dr } ds \Big)^\frac{1}{a},
\end{equation}
that we write here for future reference.

\subsection{Solvability of geometric-Bernoulli BSDEs} 
For a strategy $\alpha = (c,\pi) \in \tilde{\mathcal A}$, define the (random) coefficients
\begin{align}
\label{eq parameters Bernoulli stability random strategies}
    \varphi_t^\alpha &:= - \eta \lambda + (1-\gamma) \big( p (\pi_t \mu^1 - c_t)  -\theta  (\mu^2 - \hat b _t^2) 
    + \frac{1}{2} p (p(1-\gamma)-1) \pi_t^2 ((\nu^1)^2 +  ( \sigma^1)^2) \\ \notag 
    & \quad \quad \quad   \quad \quad  + \frac{1}{2} \theta (1 + \theta (1-\gamma)) ((\nu^2)^2 +  ( \sigma^2)^2)
   - p \theta (1-\gamma) \pi_t \sigma^1 \sigma ^2 \big), \\ \notag 
    \psi_t^\alpha & :=  \lambda \epsilon^{ -1} \big({c_t}^p {(b_t^2)^{-\theta}} \big)^{1-\frac{1}{\delta}}, \\ \notag
    \beta_t^\alpha & := \big( p(1-\gamma)\pi_t \nu^1 ,  - \theta (1-\gamma) \nu^2  , (1-\gamma) ( p \pi_t \sigma^1 - \theta (1-\gamma) \sigma^2) \big).
\end{align}
We then discuss a first characterization result for $V^\alpha$.
\begin{lemma} 
\label{lemma geometric Bernoulli equations random strategies}
{The unique solution  $(X^\alpha,Y,V^\alpha, Z^\alpha)$ of the geometric-Bernoulli BSDE \eqref{eq geometric Bernoulli BSDE X Y V} has backward component $V^\alpha$ characterized by}
$$
V_t^\alpha = \frac{ (\eta \epsilon) ^{\lambda} }{1-\gamma} H^\alpha_t  ((X^\alpha_t)^p Y _t^{-\theta}  )^{1-\gamma}, 
$$
with $(H^\alpha,N^{1, \alpha},N^{2, \alpha}, \Sigma^\alpha)$ solution to the BSDE 
\begin{align}\label{eq BSDE H}
    d H_t = & -\Psi^\alpha ( t, H_t, N^1_t,N^2_t ,\Sigma_t) dt  + (N_t^1, N_t^2,\Sigma_t) d (\tilde W ^1, \tilde W ^2 ,B)_t, \quad H_T = 1, 
\end{align}
with driver 
$\Psi^\alpha (t, h, n^1,n^2 ,z) := \varphi_t^\alpha h + \psi _t^\alpha h^{1-\frac{1}{\lambda}} +  \beta _t^\alpha (n^1, n^2, z) $ and $\varphi^\alpha , \psi^\alpha , \beta^\alpha$  defined in \eqref{eq parameters Bernoulli stability random strategies}. 
\end{lemma}

\begin{proof}
{In order to simplify the notation, we drop the superscript $\alpha$.}
The forward equations of \eqref{eq geometric Bernoulli BSDE X Y V} admits a unique solution (in explicit form) while the BSDE has a unique solution by Remark \ref{remark BSDE well defined MFG} (as $\tilde{\mathcal F}_0$ is not necessarily trivial).
Thus, we define the process $H$ as $ H_t := \frac{1-\gamma}{(\eta \epsilon )^{\lambda}} V_t  (X_t^p Y _t^{-\theta}  )^{\gamma-1}$ and we search for a BSDE representation of $H$. 

By using the terminal condition for $V_T$, it is immediate to verify that $H_T =1$.
Moreover, thanks to It\^o formula we find
\begin{align*}
    d H_t =& -\big[ \hat \varphi_t^\alpha H_t + \psi_t^\alpha H_t^{1-\frac{1}{\lambda}}  \\
    & \quad +(1-\gamma)\frac{1-\gamma}{(\eta \epsilon )^{\lambda}} (X_t^p Y _t^{-\theta}  )^{\gamma-1} 
        \big( p  \pi_t (\nu^1 Z_t^1 + \sigma^1 Z^0_t)- \theta (\nu^2 Z^2_t + \sigma^2 Z^0_t) \big) \big] dt \\    
    & \quad + \frac{1-\gamma}{(\eta \epsilon )^{\lambda}} (X_t^p Y _t^{-\theta}  )^{\gamma-1} Z_t d (\tilde W ^1, \tilde W ^2 ,B)_t \\
    & \quad + (1-\gamma) H_t \big( -p( \pi_t \nu^1 dW^1 + \pi_t \sigma^1 dB_t) + \theta (\nu^2 dW^2_t + \sigma^2 d B_t) \big), 
\end{align*}
where
\begin{align*}
    \hat \varphi_t^\alpha &:= - \eta \lambda + (1-\gamma) \big( p (\pi_t \mu^1 - c_t)  -\theta  (\mu^2 - \hat b _t^2) 
    + \frac{1}{2} p (-p(1-\gamma)-1) \pi_t^2 ((\nu^1)^2 +  ( \sigma^1)^2) \\ \notag 
    & \quad \quad \quad   \quad \quad  + \frac{1}{2} \theta (1 - \theta (1-\gamma))) ((\nu^2)^2 +  ( \sigma^2)^2)
   + p \theta (1-\gamma) \pi_t \sigma^1 \sigma ^2 \big) .
\end{align*}
Hence, defining
\begin{align*}
    N^1_t &:= -p(1-\gamma) H_t \pi_t \nu^1 + \frac{1-\gamma}{(\eta \epsilon )^{\lambda}} (X_t^p Y _t^{-\theta}  )^{\gamma-1} Z^1_t, \\
    N^2_t &:= \theta (1-\gamma) H_t \nu^2 + \frac{1-\gamma}{(\eta \epsilon )^{\lambda}} (X_t^p Y _t^{-\theta}  )^{\gamma-1} Z^2_t, \\
    \Sigma_t &:= -p(1-\gamma) H_t \pi_t \sigma^1 + \theta (1-\gamma) H_t \sigma^2 + \frac{1-\gamma}{(\eta \epsilon )^{\lambda}} (X_t^p Y _t^{-\theta}  )^{\gamma-1} Z^0_t, 
\end{align*}
and substituting into the latter equation, we obtain
\begin{align*}
    d H_t =& -\big[ \big(\hat \varphi_t^\alpha + p^2 (1-\gamma) \pi_t^2 ( (\nu^1)^2+ (\sigma^1)^2) + (1-\gamma) \theta^2 ( (\nu^2)^2+ (\sigma^2)^2) \\
     & \quad -2 p \theta (1-\gamma) \pi_t \sigma^1 \sigma ^2 \big) H_t  + \psi_t^\alpha H_t^{1-\frac{1}{\lambda}} + \beta_t^\alpha (N_t^1, N_t^2,\Sigma_t) \big] dt \\
    &\quad + (N_t^1, N_t^2,\Sigma_t) d (\tilde W ^1, \tilde W ^2 ,B)_t.
\end{align*}
Finally, by the definitions of $\varphi ^\alpha$ and of $\hat \varphi ^\alpha$, we conclude that
\begin{align*}
    d H_t =& -\big[  \varphi_t^\alpha  H_t + \psi_t^\alpha H_t^{1-\frac{1}{\lambda}} + \beta_t^\alpha (N_t^1, N_t^2,\Sigma_t) \big] dt + (N_t^1, N_t^2,\Sigma_t) d (\tilde W ^1, \tilde W ^2 ,B)_t,  
\end{align*}
which is the desired BSDE. 
\end{proof}

When the considered strategy $\alpha$ is simple, a more elementary characterization of $V^\alpha$ can be given in terms of a (random) Bernoulli ODE.
This representation also imply certain stability of the system and will be crucial when showing the convergence and approximation results (see the proofs of Theorems \ref{thm convergence} and \ref{theorem approximation} in Section \ref{section Proof of the Main Theorems}).
\begin{lemma} 
\label{lemma geometric Bernoulli equations}
For a simple strategy $\alpha = (c,\pi)$, 
{the unique solution  $(X^\alpha,Y,V^\alpha, Z^\alpha)$ of the geometric-Bernoulli BSDE \eqref{eq geometric Bernoulli BSDE X Y V} has backward component $V^\alpha$ characterized by}
$$
V^\alpha_t = \frac{(\eta \epsilon )^{\lambda} }{1-\gamma} h_{\frac 1 \lambda , \varphi^\alpha, \psi^\alpha}(t)  ((X_t^\alpha)^p Y _t^{-\theta}  )^{1-\gamma}, 
$$
with $\varphi^\alpha, \psi^\alpha$ as in \eqref{eq parameters Bernoulli stability random strategies} and $h_{\frac 1 \lambda , \varphi^\alpha, \psi^\alpha}$ as in \eqref{eq SOL Bernoulli generic}.
\end{lemma}

\begin{proof}
Thanks to Lemma \ref{lemma geometric Bernoulli equations random strategies}, $V^\alpha$ can be characterized in terms of the solution  $H^\alpha$ of  \eqref{eq BSDE H}.
Notice that if the strategy $\alpha$ is simple, then the parameters  $\varphi^\alpha, \psi^\alpha$ of \eqref{eq parameters Bernoulli stability random strategies} are $\tilde{\mathcal F}_0$-measurable. 
Hence, the BSDE \eqref{eq BSDE H} has $\tilde{\mathcal F}_0$-measurable coefficients as well as $\tilde{\mathcal F}_0$-measurable terminal condition. 
Therefore, its solution is $\tilde{\mathcal F}_0$-measurable (in particular, $N^{1,\alpha} = N^{2, \alpha} = \Sigma^\alpha =0$) and it coincides with the solution of the (random) Bernoulli ODE
$$
d H_t =  -[\varphi_t^\alpha H_t + \psi _t^\alpha H_t^{1-\frac{1}{\lambda}}  ] dt,  \quad H_T = 1, 
$$
which is given by  $h_{\frac 1 \lambda , \varphi^\alpha, \psi^\alpha}$ as in \eqref{eq SOL Bernoulli generic}.
\end{proof}

\subsection{Optimizing against simple strategies}
We next turn our focus on the optimization problem 
\begin{equation}\label{eq control problem generic}
    \Max _{\alpha \in \tilde {\mathcal A}} V_0^\alpha,
\end{equation}
where $\tilde {\mathcal A}$ is the set of strategies and $(X^\alpha,Y,V^\alpha, Z^\alpha)$ solves the system \eqref{eq geometric Bernoulli BSDE X Y V}.

Hinging on the representation of Lemma \ref{lemma geometric Bernoulli equations random strategies}, the next theorem characterizes explicitly the optimal controls and represents the starting point in order to derive the equilibria of the games (see the proofs of Theorems \ref{thm existence NE} and \ref{thmfg} in Section \ref{section Proof of the Main Theorems} below).

\begin{theorem}\label{thm optimal controls}
The control problem \eqref{eq control problem generic} admits an optimal simple  control $(c^*,\pi^*)$  given by
\begin{align}\label{eq optimal controls generic}
    c^*_t   &:= c^*_t (\mu^2, \hat b^2, \nu^2, \sigma^2, b^2; \mu^1, \nu^1, \sigma^1, f, g, p) 
                := \epsilon^{-a} (b^2_t) ^{-\theta(1-\frac{1}{\delta})a} \big( h_{\frac a \lambda, \varphi^*, \psi^*}(t)  \big)^{-\frac{a}{\lambda}} \\ \notag
    \pi^*_t &:= \pi^*_t (\mu^2, \hat b^2, \nu^2, \sigma^2, b^2; \mu^1, \nu^1, \sigma^1, f, g, p) 
        := \frac{ \mu^1 - \sigma^1 \sigma^2 \theta (1-\gamma)  }{ (1 - p (1 - \gamma))((\nu^1)^2+ (\sigma^1)^2)}, \notag
\end{align}
for functions
\begin{align}
\label{eq parameters Bernoulli control}
    \varphi^* (t) &:= - \eta \lambda + (1-\gamma) \Big[  - \theta  (\mu^2 - \hat b _t^2) + \frac{1}{2} \theta (1 + \theta (1-\gamma)) ((\nu^2)^2+ (\sigma^2)^2) \\ \notag 
    & \quad \quad \quad   \quad \quad  \quad \quad  \quad \quad    - \frac{1}{2} p \frac{( \sigma^1 \sigma^2 \theta (1-\gamma) - \mu^1 )^2}{ ( p (1-\gamma) -1)((\nu^1)^2+ (\sigma^1)^2)}\Big], \\ \notag 
    \psi^* (t) & := \epsilon^{-a } (\lambda - p (1-\gamma ))   (b_t^2)^{-\theta(1-\frac{1}{\delta})a } ,
\end{align}
and  for $a:= (1 - p (1 - \frac 1 \delta))^{-1}$.
\end{theorem}

\begin{proof}
The proof hinges on the representation of Lemma \ref{lemma geometric Bernoulli equations random strategies} and on a comparison theorem for BSDEs. 
Notice indeed that, thanks to Lemma \ref{lemma geometric Bernoulli equations random strategies}, the optimal control problem $\max_\alpha V_0^\alpha$ is equivalent to the maximization problem $\max _\alpha \frac{H_0^\alpha}{1-\gamma}$.
Hence, the optimization problem depends on the sign of $1-\gamma$, and becomes a minimization problem if $1-\gamma < 0$. 
We limit our self to show the case in which $1-\gamma >0$, the case $1-\gamma <0$ is analoguous.

We divide the rest of the proof in two steps.
\smallbreak\noindent
\emph{Step 1.}
We first consider the solution to the BSDE \eqref{eq BSDE H} with maximal driver. 
Namely, $\Psi^\alpha$ as in Lemma \ref{lemma geometric Bernoulli equations random strategies}, define 
$$
\Psi^* (t,h,n^1,n^2,z) := \sup _{\alpha \in \tilde {\mathcal A}} \Psi^\alpha (t,h,n^1,n^2,z),
$$
and observe that the $\tilde \alpha = (\tilde c,\tilde \pi)$ attaining the supremum writes as a functions of $(t,h,n^1,n^2,z)$ as
\begin{align}\label{eq optimal controls feedback}
    \tilde c (t,h,n^1,n^2,z) &:=  \epsilon^{-a } (b^2_t) ^{-a\theta(1-\frac{1}{\delta})}  h ^{-\frac{a}{\lambda}} \\ \notag 
    \tilde \pi (t,h,n^1,n^2,z) & := \frac{ \mu^1 + \nu^1 \frac{n^1}{h} + \sigma^1 \frac{z}{h} - \sigma^1 \sigma^2 \theta (1-\gamma)  }{ (1 - p (1 - \gamma))((\nu^1)^2+ (\sigma^1)^2)}. \notag
\end{align}
Moreover, $\Psi^*$ writes as 
$$
\Psi^* (t, h, n^1,n^2 ,z) = \tilde \varphi^*(t, h, n^1,n^2 ,z) h + \tilde \psi ^*(t, h, n^1,n^2 ,z) h^{1-\frac{a}{\lambda}}
$$
where we define 
\begin{align}
\label{eq parameters control volatility}
    \tilde \varphi^*(t,  h,  n^1, n^2 ,z)  &:= - \eta \lambda + (1-\gamma) \Big[  - \theta  (\mu^2 - \hat b _t^2) + \frac{1}{2} \theta (1 + \theta (1-\gamma)) ((\nu^2)^2+ (\sigma^2)^2) \\ \notag 
    & \quad \quad \quad   \quad \quad  \quad \quad  \quad \quad -\theta \nu^2 \frac{n^2}{h} - \theta \sigma^2 \frac{z}{h} \\ \notag 
    & \quad \quad \quad   \quad \quad  \quad \quad  \quad \quad    - \frac{1}{2} p \frac{( \mu^1 + \nu^1 \frac{n^1}{h} + \sigma^1 \frac{z}{h} - \sigma^1 \sigma^2 \theta (1-\gamma)  )^2}{ ( p (1-\gamma) -1)((\nu^1)^2+ (\sigma^1)^2)}\Big], \\ \notag 
    \tilde \psi^* (t, h, n^1,n^2 ,z) & := \epsilon^{-a} (\lambda - p (1-\gamma ))   (b_t^2)^{-a \theta(1-\frac{1}{\delta})}. 
\end{align}

Consider now the BSDE
\begin{equation}\label{eq BSDE H star}
d H_t =  -\Psi^* ( t, H_t, N^1_t,N^2_t ,\Sigma_t) dt  + (N_t^1, N_t^2,\Sigma_t) d (\tilde W ^1, \tilde W ^2 ,B)_t, \quad H_T = 1.
\end{equation}
Since the coefficients of $\Psi^*$ and the terminal condition are $\tilde{\mathcal F}_0$-measurable, we search for a $\tilde{\mathcal F}_0$-measurable solution $(H^*, N^{1,*}, N^{2,*}, \Sigma^*)$, with $N^{1,*} = N^{2,*} = \Sigma^* =0$.
Then, for $\varphi^*_t, \psi^*_t$ as in \eqref{eq parameters Bernoulli control}, the previous BSDE writes as
$$
d H_t =  -\Psi^* ( t, H_t, 0,0 ,0) dt =  -(\varphi^*_t H_t + \psi^*_t H^{1-\frac{a}{\lambda}}) dt,
\quad H_T = 1,
$$
which is a Bernoulli ODE with solution $H^* : = h_{\frac{a}{\lambda}, \varphi^*, \psi^*}$ (cf.\ \eqref{eq SOL Bernoulli generic}).

\smallbreak\noindent
\emph{Step 2.}
We now want to show that $H^*_0 \geq H^\alpha_0$ for any strategy $\alpha$.
To this end, we will make a logarithmic change of variable and then use a comparison principle for quadratic BSDEs. 

Define the transformation $ (\tilde h, \tilde n^1, \tilde n ^2,\tilde z ) := F( h, n^1, n^2, z)  : = ( \log h,   n^1 / h, n^2 / h, z /h)$. 
For $H^*$ as in the previous step, the process $(\tilde H, \tilde N^1, \tilde N^2, \tilde \Sigma ) := F( H^*,  0,  0, 0 )$ solves the BSDE
\begin{equation*}
d \tilde H_t =  - \tilde \Psi^* ( t, \tilde H_t, \tilde N^1_t, \tilde N^2_t ,\tilde \Sigma_t) dt  + (\tilde N_t^1, \tilde N_t^2, \tilde \Sigma_t) d (\tilde W ^1, \tilde W ^2 ,B)_t, \quad \tilde H_T = 0,
\end{equation*}
where the new driver $\tilde \Psi ^*$ is defined as
\begin{equation*}
    \tilde \Psi ^* (t,\tilde h, \tilde n^1, \tilde n^2 , \tilde z)
    := \varphi ^* (t,\tilde h, \tilde n^1, \tilde n^2 , \tilde z) + \psi ^* (t,\tilde h, \tilde n^1, \tilde n^2 , \tilde z) e^{-\frac a \lambda  \tilde h} + \frac 1 2 |(\tilde n^1, \tilde n^2 , \tilde z)|^2. 
\end{equation*}
Similarly, for generic $\alpha \in \tilde {\mathcal A}$ and $( H^\alpha,  N^{1,\alpha},  N^{2,\alpha}, \Sigma ^\alpha)$ solution to \eqref{eq BSDE H}, the process $(\tilde H ^\alpha, \tilde N^{1, \alpha}, \tilde N^{2, \alpha}, \tilde \Sigma ^\alpha) := F( H^\alpha,  N^{1,\alpha},  N^{2,\alpha}, \Sigma ^\alpha)$ solves the BSDE
\begin{equation*}
d \tilde H_t =  - \tilde \Psi^\alpha ( t, \tilde H_t, \tilde N^1_t, \tilde N^2_t ,\tilde \Sigma_t) dt  + (\tilde N_t^1, \tilde N_t^2, \tilde \Sigma_t) d (\tilde W ^1, \tilde W ^2 ,B)_t, \quad \tilde H_T = 0,
\end{equation*}
where the new driver $\tilde \Psi ^\alpha$ is defined as
\begin{align*}
    \tilde \Psi ^\alpha (t,\tilde h, \tilde n^1, \tilde n^2 , \tilde z)
    :=& \varphi ^\alpha_t + \psi ^\alpha_t e^{-\frac 1 \lambda  \tilde h}  + \beta^\alpha_t (\tilde n^1, \tilde n^2 , \tilde z) +  \frac 1 2 |(\tilde n^1, \tilde n^2 , \tilde z)|^2. 
\end{align*}

Now, the reader can easily verify that $\tilde \Psi ^*(t,\tilde h, \tilde n^1, \tilde n^2 , \tilde z) \geq \tilde \Psi ^\alpha (t,\tilde h, \tilde n^1, \tilde n^2 , \tilde z) $ for any $\alpha \in \tilde{\mathcal A}$.
Therefore, Theorem 2.6 in \cite{kobylanski2000backward} implies that
$\tilde H _0^* \geq \tilde H_0^\alpha$, which in turn gives $H _0^* = e^{\tilde H _0^*} \geq e^{ \tilde H_0^\alpha} = H_0^\alpha$, thus completing the proof in the case $1-\gamma > 0$.
\end{proof}

\section{Proof of the Main Theorems}
    \label{section Proof of the Main Theorems}

\subsection{Proof of Theorem \ref{thm existence NE}}

We search for NE involving simple strategies. 
The rest of the proof is divided into two steps.
\smallbreak \noindent 
\emph{Step 1.} 
In this step we determine the optimal control for the optimization problem of player $i$ in response to simple strategies $(c_j, \pi_j)_{j\ne i}$ chosen by its opponent.

First of all, observe that, if the opponents of player $i$ choose simple strategies $(c_j, \pi_j)_{j\ne i}$, then the process 
\begin{align*}
Y^i_t  &:= \Big( \prod_{j\ne i} X^j_t \Big)^{\frac1N} \\
       & = \Big( \prod_{j\ne i} x^j_0 \Big)^{\frac1N} \exp \Big[ \frac1N \sum_{j \ne i} \Big( \big( \pi_j \mu_j - \frac12 \pi_j^2 ( \nu_j^2 + \sigma_j^2) \big) t 
 -  \int_0^t c_j (s) ds + \pi_j \nu_j W^j_t + \pi_j \sigma_j B_t \Big)  \Big] 
\end{align*}
is an {generalized} geometric Brownian motion. 
In particular, we can write
$$
d Y^i_t = Y^i_t ( (\hat \mu _i - \hat b _i(t) ) dt + \hat \nu _i d \hat W ^i_t + \hat \sigma _i d B_t ), \quad Y^i_0 = y^i_0,
$$
where the new parameters are 
defined by
\begin{align}\label{eq parameters optimal control NE}
y^i_0 & :=  \Big( \prod_{j\ne i} x^j_0 \Big)^{\frac1N}, \quad
\hat \nu _i := \frac1N \sqrt{ \sum_{j \ne i} (\pi_j \nu _j)^2}, \quad
\hat \sigma _i := \frac1N \sum_{j \ne i} \pi_j \sigma_j, \\ \notag
\hat \mu _i &:=\frac1N \sum_{j\ne i} \big( \pi_j \mu_j  - \frac12 \pi_j^2 ( \nu_j^2 + \sigma_j^2) \big) + \frac12 ( \hat \nu _i^2 + \hat \sigma _i^2), \\ \notag
\hat b _i (t) &:= \frac1N \sum_{j \ne i} c_j(t), \quad
\hat W ^i_t := \frac{1}{\sqrt{ \sum_{j \ne i} (\pi_j \nu _j)^2}} \sum_{j\ne i} \pi_j \nu_j W^j_t. \notag
\end{align}
Moreover, the process $\hat W ^i$ is a Brownian motion independent from $W^i$ and $B$.
Thus,  set
\begin{equation*}
    \bar b _i (t) := \Big( \prod_{j\ne i} c_j (t) \Big)^{\frac1N},
\end{equation*}
and observe that the control problem of player $i$ is given by $\Max _{\alpha_i} V^i_0$, subject to
\begin{align*}
    d X^i_t &= \pi^i_t X^i_t (\mu_i dt + \nu_i dW^i_t + \sigma_i dB_t) - c_i(t) X^i_t dt, \quad X^i_0 = x^i_0, \\
    d Y^i_t &= Y^i_t ( (\hat \mu _i - \hat b _i(t) ) dt + \hat \nu _i d \hat W ^i_t + \hat \sigma _i d B_t ), \quad Y^i_0 = y^i_0, \\
    d V^i_t &= -f_i ( (c_i (t) X^i_t )^{1-\frac{\theta}{N}} (\bar b _i (t) Y^i_t)^{-\theta}, V^i_t) dt
            + Z^i_t d (W^i, \hat W^i, B)_t, 
            \  V^i_T = g_i (( X^i_T )^{1-\frac{\theta}{N}} (Y^i_T)^{-\theta} ).
\end{align*}

Since such a control problem is (for suitable choice of parameters) of type \eqref{eq control problem generic}, we can use Theorem \ref{thm optimal controls} in order to find the best response
$( c _i^*,  \pi _i^*)$ to the strategies $(c_j, \pi_j)_{j\ne i}$:
\begin{align}\label{eq optimal controls NE}
    {c}_i^*(t)   &:= {c}_t^* (\hat \mu _i, \hat b _i, \hat \nu _i,\hat \sigma _i, \bar b _i; \mu_i, \nu_i, \sigma_i, f_i, g_i, 1-{\theta_i}/{N}), \\\notag
    {\pi}_i^*(t) &:= {\pi}^*_t (\hat \mu _i, \hat b _i, \hat \nu _i,\hat \sigma _i, \bar b _i; \mu_i, \nu_i, \sigma_i, f_i, g_i, 1-{\theta_i}/{N}), 
\end{align}
where the maps ${c}_t^*$ and $\pi_t^*$ are defined in \eqref{eq optimal controls generic} and the parameters are defined in \eqref{eq parameters optimal control NE}. 

\smallbreak\noindent
\emph{Step 2.}
In this step we search for a NE $(\boldsymbol{c},\boldsymbol{\pi}) = ( (c_1,\pi_1),...,(c_N,\pi_N) )$ of the game. 
{The argument is adapted from \cite{lacker.soret2020}.}
Observe that, in light of \eqref{eq optimal controls NE}, the simple strategy profile  $(\boldsymbol{c},\boldsymbol{\pi})$ is a NE of the game if and only if it satisfies the fixed point condition
\begin{align*}
    {c}_i(t)   &= c^*_t (\hat \mu _i, \hat b _i, \hat \nu _i,\hat \sigma _i, \bar b _i; \mu_i, \nu_i, \sigma_i, f_i, g_i, 1-{\theta_i}/{N}), \\\notag
    {\pi}_i &= \pi^*_t (\hat \mu _i, \hat b _i, \hat \nu _i,\hat \sigma _i, \bar b _i; \mu_i, \nu_i, \sigma_i, f_i, g_i, 1-{\theta_i}/{N}), 
\end{align*}
where the parameters in the right hand sides are given in \eqref{eq parameters optimal control NE} as function of $(\boldsymbol{c},\boldsymbol{\pi})$.  

We first solve the fixed point for $\boldsymbol \pi$.
Setting $\Pi := \sum_{j=1}^N \sigma_j \pi_j$, the system of equations for $\pi_j$ rewrites as
$$
\pi_i = \frac{ \mu_i - \sigma_i  \theta_i (1-\gamma_i) \frac{1}{N} (\Pi - \sigma_i \pi_i)}{ (1 - (1-\frac{\theta_i}{N}) (1 - \gamma_i))(\nu_i^2+ \sigma_i^2)}, \quad i=1,...,N.
$$
Since $ (1 - (1-\frac{\theta_i}{N}) (1 - \gamma_i))(\nu_i^2+ \sigma_i^2) - \sigma_i^2 \frac{\theta_i}{N} (1-\gamma_i) \ne 0$ (by our conditions on $\gamma_i, \theta_i, \nu_i, \sigma_i$), solving for $\pi_i$ we obtain
\begin{equation}\label{eq pi NE proof}
\pi_i = \frac{ N \mu_i -  {\sigma_i} \theta_i (1-\gamma_i)  \Pi }{N (1 - (1-\frac{\theta_i}{N}) (1 - \gamma_i))(\nu_i^2+ \sigma_i^2) - \sigma_i^2\theta_i (1-\gamma_i)},
\end{equation}
so that, multiplying by $\sigma_i$ and summing over $i$, gives the equation
\begin{align*}
 \Pi = & \sum_{i=1}^N\frac{ N  \sigma_i \mu_i}{N (1 - (1-\frac{\theta_i}{N}) (1 - \gamma_i))(\nu_i^2+ \sigma_i^2) - \sigma_i^2\theta_i (1-\gamma_i)} \\
 & - \Pi  \sum_{i=1}^N \frac{ {\sigma_i^2} \theta_i (1-\gamma_i)    }{N (1 - (1-\frac{\theta_i}{N}) (1 - \gamma_i))(\nu_i^2+ \sigma_i^2) - \sigma_i^2\theta_i (1-\gamma_i)}.
\end{align*}
By our conditions on $\gamma_i, \theta_i, \nu_i, \sigma_i$, we have 
$$
1\ne - \sum_{i=1}^N \frac{ {\sigma_i^2} \theta_i (1-\gamma_i)    }{N (1 - (1-\frac{\theta_i}{N}) (1 - \gamma_i))(\nu_i^2+ \sigma_i^2) - \sigma_i^2\theta_i (1-\gamma_i)},
$$
so that the previous equation is uniquely solved by 
\begin{align*}
 \Pi = & \frac {\sum_{i=1}^N\frac{ N  \sigma_i \mu_i}{N (1 - (1-\frac{\theta_i}{N}) (1 - \gamma_i))(\nu_i^2+ \sigma_i^2) - \sigma_i^2\theta_i (1-\gamma_i)}} {1 + \sum_{i=1}^N \frac{ {\sigma_i^2} \theta_i (1-\gamma_i)    }{N (1 - (1-\frac{\theta_i}{N}) (1 - \gamma_i))(\nu_i^2+ \sigma_i^2) - \sigma_i^2\theta_i (1-\gamma_i)}} .
\end{align*}
Plugging the latter expression into \eqref{eq pi NE proof}, we obtain (after minimal computations) the formula for $\pi_i$ as in the thesis of the theorem.

We now solve the fixed point for $\boldsymbol c$.
Due to \eqref{eq optimal controls NE} (written in terms of \eqref{eq optimal controls generic}) with parameters in \eqref{eq parameters optimal control NE}, this means to solve the system of equations
\begin{align}\label{eq system first}
    {c}_i(t)  =& \epsilon_i^{ - a_i} \bar{b}_i(t)^{- {\theta_i(1-\frac{1}{\delta_i}) a_i}}  h_i (t) ^{-\frac{a_i}{\lambda_i}}, \\ \notag
     h_i '  = &
    - \big(\rho_i + \theta_i (1-\gamma_i) \hat b _i (t) \big) h_i \\ \notag
    &- \Big( \epsilon_i^{ - a_i } \big(\lambda_i - \big( 1-\frac{\theta_i}{N} \big)(1-\gamma_i )\big) \bar b_i (t) ^{-{\theta_i ( 1-\frac{1}{\delta_i} ) a_i}} \Big) h_i^{1 - \frac{a_i}{\lambda_i}} ,    \notag 
\end{align} 
for $a_i := \frac{1}{1-(1-\theta_i /n) (1 -1/{\delta_i})}$ and where the parameter $\rho_i$ is defined as
\begin{align*}
\rho_i &:= - \eta_i \lambda_i + (1-\gamma_i) \Big[  - \theta_i  \hat \mu _i  + \frac{1}{2} \theta_i (1 + \theta_i (1-\gamma_i)) (\hat \nu _i^2+ \hat \sigma _i^2) \\ \notag 
    & \quad \quad \quad   \quad \quad  \quad \quad  \quad \quad    - \frac{1}{2}  (1-\frac{\theta_i}{N}) \frac{( \sigma_i \hat \sigma _i \theta_i (1-\gamma_i) - \mu_i )^2}{ (  (1-\frac{\theta_i}{N}) (1-\gamma_i ) -1)(\nu_i^2+ \sigma_i^2)} \Big]  
\end{align*}
has already been determined (by the fixed point in $\boldsymbol \pi$). 

The first equation in \eqref{eq system first} provides an expression for $ h_i (t) ^{-\frac{a_i}{\lambda_i}}$, which can be plugged into the second equation in \eqref{eq system first} in order to obtain
\begin{equation*}
    h_i '   
    + \big(\rho_i + \theta_i (1-\gamma_i) \hat b _i (t)  
    +   \big(\lambda_i - \big( 1-\frac{\theta_i}{N} \big)(1-\gamma_i )\big) c_i (t) \big) h_i  
    =0, 
\end{equation*}
which can be rewritten in terms of $\hat b (t) := \frac1N \sum_{i=1}^N c_i (t)$ as 
\begin{equation*}
    h_i '   
    + \big(\rho_i + \theta_i (1-\gamma_i) \hat b  (t)  
    +   (\lambda_i -(1-\gamma_i )) c_i (t) \big) h_i  
    =0.
\end{equation*}
The latter differential equation, together with the terminal condition $h_i (T) = 1$, can be solved in $h_i$ giving
\begin{equation*}
    h_i(t)  
    = \exp \Big(  \int_t^T
     \big( \rho_i + \theta_i (1-\gamma_i) \hat b  (s)  
    +  (\lambda_i -(1-\gamma_i )) c_i (s) \big) ds \Big). 
\end{equation*}
Moreover, after some manipulation, the first equation in \eqref{eq system first} can be rewritten in terms of $\bar b (t) := \frac1N \sum_{i=1}^N c_i (t)$ as 
$$
h_i (t) = \epsilon_i^{-\lambda _i} c_i (t)^{-\frac{\lambda _i }{\delta _i}} \bar b (t) ^{-\lambda _i \theta _i (1 - \frac{1}{\delta_i} )}, 
$$
which plugged into the latter equation gives
\begin{equation*}
    c_i (t) 
    =  \epsilon_i^{ - \delta _i }  \bar b (t) ^{- \theta _i (\delta_i - 1 )} \exp \Big( -\frac{\delta_i}{\lambda_i} \int_t^T
     \big( \rho_i + \theta_i (1-\gamma_i) \hat b  (s)  
    +  (\lambda_i -(1-\gamma_i )) c_i (s) \big) ds \Big), 
\end{equation*}
or, equivalently, 
\begin{equation}
    \label{eqN c int c}
    c_i (t) \exp \Big(  \int_t^T c_i (s) ds \Big)
    =  \epsilon_i^{- \delta _i }   
    e^{-\frac{\delta_i}{\lambda_i} \rho_i (T-t) }
    \Big( \bar b (t) 
    \exp \Big(    \int_t^T \hat b (s) ds \Big) \Big)^{ \theta _i (1-\delta_i)}.
\end{equation}
Thus, taking the geometric average over indexes $i =1,...,N$, we have
\begin{equation*}
    \bar b (t) \exp \Big(  \int_t^T \hat b (s) ds \Big)
    = \kappa   
    e^{- K (T-t) }
    \Big( \bar b (t) 
    \exp \Big( \int_t^T \hat b (s) ds \Big) \Big)^{ \frac{1}{N} \sum_{i=1}^N \theta _i (1-\delta_i)},
\end{equation*}
where
\begin{align*}
   \kappa := \Big( \prod_{i=1}^N \epsilon_i^{- \delta _i } \Big) ^{\frac{1}{N}}, \quad
    K :=  \frac{1}{N} \sum_{i=1}^N \frac{\delta_i}{\lambda_i} \rho_i.
\end{align*}
Therefore, since $\hat q := 1-\frac{1}{N} \sum_{i=1}^N \theta _i (1-\delta_i) \ne 0$ by assumption, we obtain
\begin{equation}
\label{eqN b exp b}
    \bar b (t) \exp \Big(  \int_t^T \hat b (s) ds \Big)
    = \kappa ^{1/\hat q}  
    e^{- \frac{K}{\hat q} (T-t) }.
\end{equation}

Set now
$$
\chi_1^i:=  \epsilon_i^{- \delta _i } \kappa ^{\frac{\theta_i (1-\delta_i)}{\hat q} }
\quad \text{and} \quad
\chi_2^i := - \frac{\delta_i}{\lambda_i} \rho_i - \theta_i (1-\delta_i) \frac{K}{\hat q},
$$
plug \eqref{eqN b exp b} into \eqref{eqN c int c} in order to obtain
$$
c_i (t) \exp \Big(  \int_t^T c_i (s) ds \Big) = \chi_1^i e^{ \chi_2^i(T-t)}.
$$
Integrating and then computing the logarithm, we obtain
$$
 \int_t^T c_i (s)ds = 
\begin{cases}
\log \big[ 1 + \frac{\chi_1^i}{\chi_2^i} ( e^{ \chi_2^i (T-t)} -1 ) \big], &  \text{ if } \chi_2^i \ne 0,\\
\log [ \chi_1^i (T-t) + 1 ], &  \text{ if }\chi_2^i = 0.
\end{cases}
$$
Finally, taking the derivative, we have
\begin{align*}
c_i(t) = 
  \begin{cases} 
      \Big( \frac{1}{{\chi_2^i}} + \big( \frac{1}{\chi_1^i} - \frac{1}{\chi_2^i} \big)e^{ - \chi_2^i(T-t)}  \Big)^{-1} & \text{ if }\chi_2^i\neq 0, \\
      \big(T-t+\frac{1}{\chi_1^i}\big)^{-1} & \text{ if } \chi_2^i= 0,
   \end{cases}  
\end{align*}
thus completing the proof of the theorem.

\subsection{Proof of Theorem \ref{thmfg}}
The proof is similar to the proof of Theorem \ref{thm existence NE}, we provide a sketch for the sake of completeness. 

We first write the parameters of the control problem of the representative player, optimizing against a population of players using a simple (thus, $\mathcal F _0$-measurable) strategy $\alpha  = (c, \pi)$.
Indeed, the resulting state equation is given by
$$
d Y_t = Y_t ( (\hat \mu - \hat b (t) ) dt  + \hat \sigma  d B_t ), \quad Y_0 = y_0,
$$
where the new parameters are defined by
\begin{align}\label{eq parameters optimal control MFGE}
y_0 & :=  \exp( \mathbb E [ \log x_0 ]), \quad
\hat \nu := 0, \quad
\hat \sigma := \mathbb E [\pi \sigma ], \\ \notag
\hat \mu &:= \mathbb E [\pi \mu]  - \frac12 \big( \mathbb E [\pi^2 ( \nu^2 + \sigma^2) ] - (\mathbb E [ \pi \sigma ] )^2 \big), \\ \notag
\hat b _t &:= \mathbb E[c_t], \quad
\bar b _t := \exp( \mathbb E [ \log c_t]). \notag
\end{align}
Thus, the control problem of the representative player is given by $\Max _{\alpha} V_0$, subject to
\begin{align*}
    d X_t &= \pi_t X_t (\mu dt + \nu dW_t + \sigma dB_t) - c(t) X_t dt, \quad X_0 = x_0, \\
    d Y_t &= Y_t ( (\hat \mu  - \hat b (t) ) dt  + \hat \sigma  d B_t ), \quad Y_0 = y_0, \\
    d V_t &= -f ( c (t) X_t  (\bar b  (t) Y_t)^{-\theta}, V_t) dt
            + Z_t d (W, B)_t,             \  V_T = g ( X_T  Y_T^{-\theta} ),
\end{align*}
which is of type \eqref{eq control problem generic}. 
Thanks to Theorem \ref{thm optimal controls}, the MFGE $(c, \pi)$ satisfies the relation:
\begin{align}\label{eq optimal controls MFGE}
    {c}(t)   &:= {c}_t^* (\hat \mu , \hat b , 0,\hat \sigma , \bar b ; \mu, \nu, \sigma, f, g, 1), \\\notag
    {\pi}(t) &:= {\pi}^*_t (\hat \mu , \hat b , 0,\hat \sigma , \bar b ; \mu, \nu, \sigma, f, g, 1), 
\end{align}
where the maps ${c}_t^*$ and $\pi_t^*$ are defined in \eqref{eq optimal controls generic} and the parameters are defined in \eqref{eq parameters optimal control MFGE}. 

We then search for a fixed point. Solving for $\pi$ first, we easily obtain
\begin{equation*}
\pi = \frac{\mu}{\gamma (\sigma^2+\nu^2)}- \theta \frac{1-\gamma}{\gamma} \frac{\sigma}{(\sigma^2+\nu^2)}\left[\frac{\mathbb{E}\left[\frac{\mu \sigma}{\gamma(\sigma^2+\nu^2)}\right]}{1+\mathbb{E}\left[\theta(\frac{1}{\gamma}-1)(\frac{\sigma^2}{\sigma^2+\nu^2})\right]}\right].
\end{equation*}
We next solve for $c$.
By using \eqref{eq optimal controls generic}, we write the system
\begin{align}\label{eq system first MFG}
    {c}(t)  =&  \epsilon^{-\delta} \bar{b}(t)^{\theta(1-\delta)}  h (t) ^{-\frac{\delta}{\lambda}}, \\ \notag
     h '  = &
    - \big(\rho + \theta (1-\gamma) \hat b  (t) \big) h - \big( \epsilon^{-\delta} \frac{\lambda}{\delta}\bar b (t) ^{\theta (1-\delta)} \big) h^{1 - \frac{\delta}{\lambda}} ,    \notag 
\end{align} 
where the parameter
\begin{align*}
\rho &:= - \eta \lambda + (1-\gamma) \Big[  - \theta  \hat \mu   + \frac{1}{2} \theta (1 + \theta (1-\gamma))  \hat \sigma ^2   + \frac{1}{2} \frac{( \sigma \hat \sigma  \theta (1-\gamma) - \mu )^2}{ \gamma(\nu^2+ \sigma^2)} \Big]  
\end{align*}
has already been determined (by the fixed point in $\pi$). 
Solving the first equation in \eqref{eq system first MFG} for $ h (t) ^{-\frac{\delta}{\lambda}}$, then plugging into the second equation and integrating the resulting differential equation, we obtain \begin{equation*}
    h(t)  
    = \exp \Big(  \int_t^T
     \big( \rho + \theta (1-\gamma) \hat b (s)  
    +  \frac{\lambda}{\delta} c (s) \big) ds \Big). 
\end{equation*}
Sobstituting back into the first equation in \eqref{eq system first MFG}, we obtain
\begin{equation*}
    c (t) \exp \Big(  \int_t^T c (s) ds \Big)
    = \epsilon^{-\delta}   
    e^{-\frac{\delta}{\lambda} \rho (T-t) } 
    \exp \Big( \theta  (1-\delta) \big( \mathbb E[ \log c_t] + \int_t^T \mathbb E [ c_s ] ds \big) \Big). 
\end{equation*}
Taking expectations, we can solve for
$ \mathbb E[ \log c_t] + \int_t^T \mathbb E [ c_s ] ds $, from which we obtain
$$
c (t) \exp \Big(  \int_t^T c (s) ds \Big) = \chi_1 e^{ - \chi_2(T-t)},
$$
with 
\begin{align*}
\chi_1 &:=  \epsilon^{-\delta} \exp \Big( \frac{\theta (1-\delta)}{1- \mathbb E [\theta (1-\delta) ]} \mathbb E [ -\delta \log \epsilon ] \Big), \\
\chi_2 &:=  \frac{\delta}{\lambda} \rho  +  \frac{\theta (1-\delta)}{1- \mathbb E [\theta (1-\delta) ]} \mathbb E \big[\frac{\delta}{\lambda} \rho \big], 
\end{align*}
Integrating the latter equation, then computing the logarithm and then taking the derivative, we have
\begin{align*}
c_t = 
  \begin{cases} 
      \Big( \big( \frac{1}{\chi_1} + \frac{1}{\chi_2} \big)e^{\chi_2(T-t)} - \frac{1}{{\chi_2}}\Big)^{-1} & \text{ if }\chi_2\neq 0, \\
      \big(T-t+\frac{1}{\chi_1}\big)^{-1} & \text{ if } \chi_2= 0,
   \end{cases}  
\end{align*}
thus completing the proof of the theorem.

\subsection{Proof of Theorem \ref{thm convergence}}
We divide the proof in two steps. 

\smallbreak\noindent
\emph{Step 1.}
We first study the convergence of the equilibrium strategies and of the mean field terms.

Using the explicit expressions derived in Theorems \ref{thm existence NE} and \ref{thmfg}, elementary computations show that
\begin{equation}\label{eq convergence pi}
| \hat \pi _{i,N} - \hat \pi | = O(1/N).
\end{equation}
Moreover, in the symmetric case we have 
\begin{equation*}\label{eq convergence chi 1}
 \chi _1^{i,N} = \chi_1, \quad \text{for any $i$ and $N$,}
\end{equation*}
and, using \eqref{eq convergence pi}, one obtains $|  \rho _{i,N} -  \rho |= O(1/N)$, which in turn gives
\begin{equation*}\label{eq convergence chi 2}
 | \chi _2^{i,N} - \chi_2| = O(1/N), \quad \text{for any $i$ and $N$.}
\end{equation*}
Thus, from the latter two equations we conclude that
\begin{equation}\label{eq convergence c}
\sup_{t \in [0,T] } | \hat c _{i,N}(t) - \hat c (t) | = O(1/N).
\end{equation}
Furthermore, since $\hat c $ is deterministic we have $\hat m _t = \hat c _t$ and. since $\hat c _{i,N}$ does not depend on $i$ and it is deterministic, we obtain
\begin{equation*}
\sup_{t \in [0,T] } | \bar c ^N(t) - \hat m (t) | = O(1/N).
\end{equation*}

Next, for generic $i$ we can write
\begin{align*}
\overline{X}^N_t  = {\big( \Pi_{j=1}^N x_0^j \big)^{\frac1N}}   \exp  &\Big( \big( \hat \pi_{i,N} \mu - \frac12 \hat \pi_{i,N}^2 ( \nu^2 + \sigma^2) \big) t \\
& \quad 
 -  \int_0^t \hat c _{i,N} (s) ds + \nu \hat \pi _{i,N} \frac1N \sum_{j=1}^N W^j_t + \hat \pi _{i,N} \sigma B_t \Big), 
\end{align*}
as well as
$$
\hat{Y}_t  = {\exp( \mathbb E [ \log x_0 ])}  \exp  \Big( \big( \hat \pi \mu - \frac12 \hat \pi^2 ( \nu^2 + \sigma^2) \big) t 
 -  \int_0^t \hat c  (s) ds + \pi \sigma B_t \Big). 
$$
In light of \eqref{eq convergence pi} and \eqref{eq convergence c}, from the strong law of large numbers (see Theorem 5.29 at p.\ 122 in \cite{Klenke13}) it follows that, for any $a>0$, one has
$$
 \sup_{t \in [0,T]} |  \overline{X}_t^N - \hat Y _t| = O\big(N^{-\frac12} (\log N)^{-\frac12 -a} \big),  \text{ $\mathbb P$-a.s.,}
$$
as desired.

\smallbreak\noindent
\emph{Step 2.}
We next study the limit of $V^i_0 (\hat{\boldsymbol{\alpha}}_N )$ using the representation of Lemma \ref{lemma geometric Bernoulli equations}. 

Fix $i \in \{1,...,N\}$. 
In order to write  $V^i_0 (\hat{\boldsymbol{\alpha}}_N )$, notice that it corresponds to the backward component of a system of type \eqref{eq geometric Bernoulli BSDE X Y V}, with forward components
\begin{align*}
d X^i_t &= \hat \pi _{i,N} X^i_t (\mu +  \nu d W^i_t + \sigma d B _t) -\hat c_{i,N} (t) X^i_t dt, \quad X^i_0 = x_0^i, \\ 
d Y^{i,N}_t &= Y^{i,N}_t ( (\hat \mu _N - \hat b _N(t)  ) dt + \hat \nu _N d \hat W ^{i,N}_t + \hat \sigma _N d B_t ), \quad Y^{i,N}_0 = y^{i,N}_0,
\end{align*}
and parameters defined by
\begin{align*}
y ^{i,N}_0 & :=   \big( \Pi_{j\ne i} x_0^j \big)^{\frac{1}{N}}, \quad
\hat \nu _N := \frac{\sqrt{N-1 }}{N } \hat \pi _{i,N} \nu, \quad
\hat \sigma _N := \frac{N-1 }{N } \hat \pi _{i,N} \sigma, \\ \notag
\hat \mu _N &:=\frac{N-1 }{N } \big( \hat \pi _{i,N} \mu  - \frac12 \hat \pi_{i,N}^2 ( \nu^2 + \sigma^2) \big) + \frac12 ( \hat \nu _N^2 + \hat \sigma _N^2), \\ \notag
\hat b _N(t) &:= \frac{N-1 }{N } \hat c _{i,N} (t), \quad
\bar b _N(t) := \hat c _{i,N} (t) ^\frac{N-1 }{N } , \quad
\hat W ^{i,N}_t := \frac{1}{\sqrt{N-1}} \sum_{j\ne i} W^j_t, \notag
\end{align*}
with $p_N = 1 - \theta /N$.
Thus, Lemma \ref{lemma geometric Bernoulli equations} gives
$$
V^i_0 (\hat{\boldsymbol{\alpha}}_N ) = \frac{(\eta \epsilon)^\lambda }{1-\gamma} h_{\frac 1 \lambda , \hat \varphi _N, \hat \psi _N}(t)  \Big( x_0^i   \big( \Pi_{j=1}^N x_0^j \big)^{- \frac{\theta}{N}} \Big)^{1-\gamma}, 
$$
where $h_{\frac 1 \lambda , \hat \varphi _N, \hat \psi _N}$ is given by \eqref{eq SOL Bernoulli generic} with 
\begin{align*}
     \hat \varphi _N (t) &:= - \eta \lambda + (1-\gamma) \big( p_N (\hat \pi _{i,N} \mu - \hat c_{i,N} (t) ) \\ \notag 
    & \quad \quad \quad   \quad \quad -\theta  (\hat \mu_N - \hat b _N(t)) 
    + \frac{1}{2} p_N (p_N (1-\gamma)-1) \hat \pi_{i,N}^2 (\nu^2 +   \sigma^2) \\ \notag 
    & \quad \quad \quad   \quad \quad  + \frac{1}{2} \theta (1 + \theta (1-\gamma)) (\hat \nu _N^2 + \hat \sigma _N^2)
   - p_N \theta (1-\gamma) \hat \pi_{i,N}  \sigma \hat \sigma_N \big), \\ \notag 
    \hat \psi _N (t) & :=  \lambda \epsilon^{ -1} \big({\hat c_{i,N}(t)}^{p_N} {(\bar b_N (t) )^{-\theta}} \big)^{1-\frac{1}{\delta}}.
\end{align*}

On the other hand, 
 $V_0 (\hat{\alpha}, \hat m, \hat Y)$ corresponds to the backward component of a system of type \eqref{eq geometric Bernoulli BSDE X Y V}, with forward components
 \begin{align*}
    d X_t &= \hat \pi X_t (\mu dt + \nu dW_t + \sigma dB_t) - \hat c _t X_t dt, \quad X_0 = x_0, \\
    d \hat Y_t &= \hat Y_t ( (\hat \mu  - \hat b (t) ) dt  + \hat \sigma  d B_t ), \quad \hat X_0 = y_0,
\end{align*}
where the  parameters are given by
\begin{align*}
y_0 & :=   \exp( \mathbb E [ \log x_0 ]), \quad
\hat \nu := 0, \quad
\hat \sigma := \hat \pi \sigma , \\ \notag
\hat \mu &:= \hat \pi \mu   - \frac12  \hat \pi^2 \nu^2 , \\ \notag
\hat b _t &:=\bar b _t := \hat c_t, \notag
\end{align*}
with $p=1$.
Thus, we have
$$
V_0 (\hat{\alpha}, \hat m, \hat Y) = \frac{(\eta \epsilon) ^{\lambda} }{1-\gamma} h_{\frac 1 \lambda , \hat \varphi, \hat \psi }(t)  (x_0 \exp( -\theta \mathbb E [ \log x_0 ]) )^{1-\gamma}, 
$$
where 
\begin{align*}
    \hat \varphi_t &:= - \eta \lambda + (1-\gamma) \big( (\hat \pi \mu - \hat c_t)  -\theta  (\hat \mu - \hat b _t) 
    - \frac{1}{2}  \gamma \hat \pi^2 (\nu^2 +   \sigma^2) \\ \notag 
    & \quad \quad \quad   \quad \quad  + \frac{1}{2} \theta (1 + \theta (1-\gamma)) \hat \sigma^2
   - \theta (1-\gamma) \hat \pi \sigma \hat\sigma \big), \\ \notag 
    \hat \psi_t & :=  \lambda \epsilon^{ -1} \big(\hat c_t ( \bar b_t)^{-\theta} \big)^{1-\frac{1}{\delta}}.
\end{align*}

Finally, from \eqref{eq convergence c} and \eqref{eq convergence pi}, we find
$$
 \sup_{t \in [0,T]} \big( | \hat \varphi _N (t) - \hat \varphi (t)| + |\hat \psi _N(t) - \hat \psi (t)| \big) = O(1/N),
$$
which in turns implies that
$$
 \sup_{t \in [0,T]} \big| h_{\frac 1 \lambda , \hat \varphi _N, \hat \psi _N}(t) - h_{\frac 1 \lambda , \hat \varphi, \hat \psi }(t) \big| = O(1/N).
$$
The latter equation allows to conclude that 
$$
{ | \mathbb E [  \log V^i_0 (\hat{\boldsymbol{\alpha}}_N ) - \log V_0 (\hat \alpha , \hat m, \hat Y) ]  | = O(1/N), }
$$
thus completing the proof.

\subsection{Proof of Theorem \ref{theorem approximation}} 
We divide the proof in three steps.
\smallbreak\noindent
\emph{Step 1.}
For any fixed $i\in \{1,...,N\}$, we want to show  that 
\begin{equation}\label{eq approximation first}
V_0^i (\hat{\boldsymbol \alpha}^\infty_N) - \sup_\alpha V_0^i (\alpha, \hat{\boldsymbol \alpha} ^\infty_{-i, N})\to 0, 
\quad \text{as $N \to \infty$,}
\end{equation}
and determine the rate of convergence. 

Consider the optimization problem of player $i$ against $\hat{\boldsymbol \alpha} ^\infty_{-i, N}$. 
Similarly to Step 1 in the proof of Theorem \ref{thm existence NE}, notice that such an optimization problem is of type \eqref{eq control problem generic}, in which player $i$ optimizes against the geometric Brownian motion
$$
d Y^{i,N}_t = Y^{i,N}_t ( (\hat \mu _N - \hat b _N(t) ) dt + \hat \nu _N d \hat W ^{i,N}_t + \hat \sigma _N d B_t ), \quad Y^{i,N}_0 = y^{i,N}_0,
$$
where the parameters are defined by
\begin{align}\label{eq parameters approx}
y ^{i,N}_0 & :=  \Big( \prod_{j\ne i} x^j_0 \Big)^{\frac1N}, \quad
\hat \nu _N := \frac{\sqrt{N-1 }}{N } \hat \pi \nu, \quad
\hat \sigma _N := \frac{N-1 }{N } \hat \pi \sigma, \\ \notag
\hat \mu _N &:=\frac{N-1 }{N } \big( \hat \pi \mu  - \frac12 \hat \pi^2 ( \nu^2 + \sigma^2) \big) + \frac12 ( \hat \nu _N^2 + \hat \sigma _N^2), \\ \notag
\hat b  _N(t) &:= \frac{N-1 }{N } \hat c(t), \quad
\bar b  _N(t) := \hat c(t) ^\frac{N-1 }{N } , \quad
\hat W ^{i,N}_t := \frac{1}{\sqrt{N-1}} \sum_{j\ne i} W^j_t, \notag
\end{align}
with $p_N := 1-\theta/N$.
Using Theorem \ref{thm optimal controls}, the optimal response $\alpha_{i,N}^* = (c_{i,N}^*, \pi_{i,N}^*)$ of player $i$ is given by  
\begin{align*}
    {c}_{i,N}^* (t)   &:= c_t^* (\hat \mu _N, \hat b _N, \hat \nu _N, \hat \sigma _N , \bar b _N; \mu, \nu, \sigma, f, g, 1-{\theta}/{N}), \\\notag
    {\pi}_{i,N}^* &:= \pi_t^* (\hat \mu _N, \hat b _N, \hat \nu _N, \hat \sigma _N , \bar b _N; \mu, \nu, \sigma, f, g, 1-{\theta}/{N}),  
\end{align*}
and \eqref{eq approximation first} becomes equivalent to the limit
\begin{equation}\label{eq approximation second}
V_0^i (\hat{\boldsymbol \alpha}^\infty_N) -  V_0^i ( \alpha_{i,N}^*, \hat{\boldsymbol \alpha} ^\infty_{-i, N})\to 0, 
\quad \text{as $N \to \infty$,} 
\end{equation}
that we will investigate in the next steps.
\smallbreak\noindent
\emph{Step 2.}  
In this step we study the limit of $\alpha_{i,N}^*$ as $N \to \infty$.

First of all, 
the optimal control $\alpha_{i,N}^* = (c_{i,N}^*, \pi _{i,N}^*)$ can be written explicitly as
\begin{equation}
    \label{eq writing response against MFGE}
    \pi  _{i,N}^*= \tilde a_N \frac{\mu- \sigma \hat \sigma _N \theta (1-\gamma)}{ (\nu^2 + \sigma ^2) }, 
    \quad
    c _{i,N}^*(t) =  \varepsilon^{-a_N } \hat c (t)^{-\frac{N-1}{N}  \theta ( 1-\frac1\delta) a_N } \hat h^*_N (t) ^{-\frac{a_N}{\lambda}},
\end{equation}
where 
\begin{align*} 
\tilde a _N := \begin{matrix}(\gamma + \frac{\theta}N(1-\gamma))^{-1}\end{matrix}, \quad
a_N := \begin{matrix} (\frac 1 \delta + \frac{\theta}N(1- \frac 1 \delta ))^{-1} \end{matrix},
\end{align*}
$h^*_N = h_{ \frac{a_N}{\lambda}, \varphi_N^*, \psi_N^*}$
and
\begin{align*}
     \varphi^*_N (t) &:=  - \eta \lambda + (1-\gamma) \Big[  - \theta  \big(\hat \mu_N -\frac{N-1}{N} \hat c(t) \big) + \frac{1}{2} \theta (1 + \theta (1-\gamma)) (\hat \nu _N^2  + \hat \sigma _N^2 ) \\ 
    & \quad \quad \quad   \quad \quad  \quad \quad  \quad \quad    + \frac{1}{2}  \big(1-\frac {\theta} N \big) \frac{ \tilde a _N ( \sigma \hat \sigma_N \theta (1-\gamma) - \mu )^2}{ (\nu^2 + \sigma ^2 )}  \Big], \\ \notag 
    \psi^*_N  (t) & :=   \epsilon^{-a_N } \big(\lambda -  \big(1-\frac {\theta} N \big)(1-\gamma )\big)   \hat c (t)^{-\frac{N-1}{N}  \theta ( 1-\frac1\delta) a_N } .
\end{align*}
Secondly, 
using the optimality condition in the definition of MFGE, we have
\begin{equation}
    \label{eq writing MFGE}
    \hat \pi = \frac{\mu - \hat \pi \sigma^2 \theta (1-\gamma) }{\gamma (\nu^2 + \sigma^2)}
    \quad \text{and} \quad
    \hat c (t) =  \epsilon^{-\delta} \hat c  (t)^{\theta (1-\delta) } \hat h (t) ^{-\frac{\delta}{\lambda}},
\end{equation}
where $\hat h = h_{\frac \delta \lambda, \hat  \varphi, \hat \psi }$ and
\begin{align*}
    \hat \varphi (t) &:= - \eta \lambda + (1-\gamma) \Big[  - \theta  (\hat \pi \mu - \hat c(t)) + \frac{1}{2} \theta (1 + \theta (1-\gamma)) \hat \pi ^2 \sigma^2 \\ 
    & \quad \quad \quad   \quad \quad  \quad \quad  \quad \quad    + \frac{1}{2}  \frac{( \hat \pi \sigma^2 \theta (1-\gamma) - \mu )^2}{ \gamma (\nu^2+ \sigma^2) } \Big], \\ \notag 
    \hat \psi (t) & := \epsilon^{-\delta} (\lambda -  (1-\gamma ))   (\hat c (t))^{\theta (1-\delta)} .
\end{align*}
Thus,
from \eqref{eq writing MFGE} and \eqref{eq writing response against MFGE} we find
\begin{equation}\label{eq limits pi}
| \pi^*_{i,N} - \hat \pi | = O(1 / N ).
\end{equation}
Moreover, we also find
$$
 \sup_{t \in [0,T]} \big( | \varphi _N^*(t) - \hat \varphi (t)| + |\psi _N^*(t) - \hat \psi (t)| \big) = O(1/N),
$$
which in turns implies that
$$
 \sup_{t \in [0,T]} | h _N^*(t) - \hat h (t)| = O(1/N).
$$
and
\begin{equation}\label{eq limits c}
 \sup_{t \in [0,T] } |c_{i,N}^* (t) - \hat c (t)| = O(1/N).
\end{equation}
\smallbreak\noindent
\emph{Step 3.} 
In this step we will employ the limits in \eqref{eq limits pi} and \eqref{eq limits c} together with the representation of Lemma \ref{lemma geometric Bernoulli equations} in order to conclude the proof. 

By Lemma \ref{lemma geometric Bernoulli equations} we have the representations
\begin{align}\label{eq rappr V approx} 
V^i_0(\hat{\boldsymbol \alpha}^\infty_{ N}) &= \frac{(\eta \epsilon)^{\lambda} }{1-\gamma} \tilde h _N (0) \Big( \frac{x_0^i}{ (y ^{i,N}_0 )^\theta} \Big)^{1-\gamma}, 
\\ \notag
V_0^i (\alpha^*_{i,N}, \hat{\boldsymbol \alpha}^\infty_{-i, N}) & = \frac{(\eta \epsilon)^{\lambda} }{1-\gamma} \tilde h^*_N(0) \Big( \frac{x_0^i}{ (y ^{i,N}_0)^\theta} \Big)^{1-\gamma},  \notag
\end{align} 
where $\tilde h _N = h_{\frac 1 \lambda , \tilde \varphi _N, \tilde \psi _N}$, with
\begin{align*}
 \tilde \varphi _N (t) &= - \eta \lambda + (1-\gamma) \Big[ \big(1-\frac{\theta}{N}\big) (\hat \pi \mu - \hat c_t) 
 -\theta  \big(\hat \mu_N - \frac{N-1}{N} \hat c _t \big) \\ \notag
    & \quad \quad \quad   \quad \quad+ \frac{1}{2} \big( 1-\frac{\theta}{N} \big) \big( \big(1-\frac{\theta}{N}\big)(1-\gamma)-1 \big) \hat \pi ^2 (\nu^2 +  \sigma^2) \\ \notag 
    & \quad \quad \quad   \quad \quad  + \frac{1}{2} \theta (1 + \theta (1-\gamma)) (\hat \nu _N^2 +  \hat \sigma_N^2)
   - \big(1-\frac{\theta}{N}\big) \theta (1-\gamma) \hat \pi \sigma \hat \sigma_N \Big], \\ \notag 
\tilde \psi _N (t) & = \lambda \epsilon^{-1}  \hat c_t ^{(1-\theta) (1-\frac{1}{\delta})}, 
\end{align*}
and  $ \tilde h ^*_N = h_{\frac 1 \lambda , \tilde \varphi ^*_N, \tilde \psi ^*_N}$, with
\begin{align*}
\tilde \varphi ^*_N (t) &= - \eta \lambda + (1-\gamma) \Big[ \big(1-\frac{\theta}{N}\big) (\pi^*_{i,N} \mu - c^*_{i,N} (t)) 
 -\theta  \big(\hat \mu_N - \frac{N-1}{N} \hat c _t \big) \\ \notag
    & \quad \quad \quad   \quad \quad + \frac{1}{2} \big( 1-\frac{\theta}{N} \big) \big( \big(1-\frac{\theta}{N}\big)(1-\gamma)-1 \big) (\pi^*_{i,N}) ^2 (\nu^2 +  \sigma^2) \\ \notag 
    & \quad \quad \quad   \quad \quad  + \frac{1}{2} \theta (1 + \theta (1-\gamma)) (\hat \nu _N^2 +  \hat \sigma_N^2)
   - \big(1-\frac{\theta}{N}\big) \theta (1-\gamma) \pi^*_{i,N} \sigma \hat \sigma_N \Big], \\ \notag 
\tilde \psi ^*_N (t) & = \lambda \epsilon^{-1}  \big( c^*_{i,N} (t) ^{1-\frac{\theta}{N}} \hat c_t ^{-\frac{N-1}{N}\theta} \big)^{1-\frac{1}{\delta}}.
\end{align*}
Thanks to \eqref{eq limits pi} and \eqref{eq limits c}, taking limits into the latter two equations, we obtain 
$$
 \sup_{t \in [0,T]} \big( |\tilde \varphi _N(t) - \tilde \varphi ^*_N(t)| + |\tilde \psi _N(t) - \tilde \psi ^*_N(t)| \big) = O(1/N),
$$
which in turns implies that
$$
 |\tilde h _N(0) - \tilde h ^*_N(0)|  = O(1/N).
$$
The latter limits, together with \eqref{eq rappr V approx}, allow  to conclude that
\begin{align} 
    |V^i_0(\hat{\boldsymbol \alpha}^\infty_{ N}) - V_0^i (\alpha^*_{i,N}, \hat{\boldsymbol \alpha}^\infty_{-i, N})|
     =  O(1/N), \quad \mathbb P \text{-a.s.,}
\end{align} 
which prove the convergence in \eqref{eq approximation second} (hence in \eqref{eq approximation first}) with the desire rate. 
This completes the proof.

\section{Conclusion}\label{endiscuss}
Our paper solves games with relative performance concerns and Epstein-Zin recursive preferences. We have seen that assuming time-additive preferences can lead to substantially different conclusions as these preferences do not allow to differentiate risk aversion and elasticity of intertemporal substitution.   

\smallskip 
\textbf{Acknowledgements.} 
Funded by the Deutsche Forschungsgemeinschaft (DFG, German Research Foundation) - Project-ID 317210226 - SFB 1283.

\bibliographystyle{siam} 
\bibliography{main.bib}

\vspace{20pt}

\noindent \textbf{Jodi Dianetti}:
Center for Mathematical Economics (IMW), \\
Bielefeld University, \\
Universitätsstrasse 25, 33615, Bielefeld, Germany \\
Email: \url{jodi.dianetti@uni-bielefeld.de}

\vspace{10pt} % Add some vertical space between the entries

\noindent \textbf{Frank Riedel}:
Center for Mathematical Economics (IMW), \\
Bielefeld University, \\
Universitätsstrasse 25, 33615, Bielefeld, Germany \\
Email: \url{frank.riedel@uni-bielefeld.de}

\vspace{10pt} % Add some vertical space between the entries

\noindent \textbf{Lorenzo Stanca}: 
Collegio Carlo Alberto and University of Turin, Department of  \\ ESOMAS,
Corso Unione Sovietica, 218 Bis, 10134, Turin, Italy \\
Email: \url{lorenzomaria.stanca@unito.it}
\end{document}